\documentclass[11pt]{amsart}
\usepackage{amsmath,amssymb}
\usepackage{color}
\usepackage[all,curve,frame]{xy}

\usepackage[pdftex]{graphicx}

\newtheorem{defi}{Definition}[section]
\newtheorem{sinobservacion}[defi]{}
\newenvironment{sinob}{\begin{sinobservacion} \rm}{\end{sinobservacion} }
\newtheorem{coro}[defi]{Corollary}
\newtheorem{lema}[defi]{Lemma}

\newtheorem{obs}[defi]{Remark}
\newtheorem{prop}[defi]{Proposition}
\newtheorem{teo}[defi]{Theorem}
\newtheorem{ej}[defi]{Example}

\newcommand{\eps}{\varepsilon}

\newcommand{\benu}{\begin{enumerate}}
\newcommand{\enu}{\end{enumerate}}

\newcommand{\al}{\alpha}

\newcommand{\be}{\beta}

\newcommand{\mbmA}{\mbox{mod}\,A}

\newcommand{\mbiA}{\mbox{ind}\,A}

\newcommand{\fepi}{\twoheadrightarrow}

\newcommand{\Gaa}{\Gamma_A}

\newcommand{\fle}{\rightarrow}

\begin{document}
\title[On the nilpotency index of the radical of a module category]
{On the nilpotency index of the radical of a module category}
\author[Chaio]{Claudia Chaio}
\address{Centro marplatense de Investigaciones Matem\'aticas, Facultad de Ciencias Exactas y
Naturales, Funes 3350, Universidad Nacional de Mar del Plata, 7600 Mar del
Plata, CONICET Argentina}
\email{claudia.chaio@gmail.com}

\author[Guazzelli]{Victoria Guazzelli}
\address{Centro marplatense de Investigaciones Matem\'aticas, Facultad de Ciencias Exactas y
Naturales, Funes 3350, Universidad Nacional de Mar del Plata, 7600 Mar del
Plata, Argentina}
\email{victoria.guazzelli@gmail.com}

\author[Suarez]{Pamela Suarez}
\address{Centro marplatense de Investigaciones Matem\'aticas, Facultad de Ciencias Exactas y
Naturales, Funes 3350, Universidad Nacional de Mar del Plata, 7600 Mar del
Plata, Argentina}
\email{pamelaysuarez@gmail.com}

\maketitle

\begin{abstract}
Let $A$ be a finite dimensional representation-finite algebra over
an algebraically closed field. The aim of this  work is to determine which vertices of $Q_A$ are sufficient  to be consider in order to compute the nilpotency index of the radical of the module category of $A$. In many cases, we give a formula to compute such index  taking into account the ordinary quiver of the given algebra.
\end{abstract}

\section*{Introduction}
Let $A$ be a finite dimensional $k$-algebra over an algebraically closed field $k$, and consider $\mbox{mod}\,A$ to
be the category of finitely generated right $A$-modules. For $X,Y$ indecomposable $A$-modules,
we denote by $\Re(X,Y)$ the set of
all non-isomorphisms $f: X \rightarrow Y$. Inductively, the powers of
$\Re(X,Y)$ are defined, see (\ref{rad}).

In case we deal with a representation-finite algebra, it is well-known by a result of M. Auslander
that there is a positive integer $n$ such that $\Re^{n}(\mbox{mod}\,A)=0$, see \cite[p. 183]{ARS}.
The minimal lower bound $m \geq 1$ such that $\Re^{m}(\mbox{mod}\,A)$ vanishes
is called the nilpotency index of $\Re(\mbmA)$.

In \cite{C},  the first named author studied the problem to determine the nilpotency index of the radical of $\mbox{mod}\, A$ for $A$ a representation-finite algebra, in case $A$ is a finite dimensional algebra over an algebraically closed field. She proved that the
nilpotency index of $\Re(\mbmA)$ can be obtained  in terms of the left and right degrees of some particular irreducible morphisms. The concept of degree of an irreducible morphism was introduced by S. Liu in \cite{L} and has been a powerful tool to solve many problems concerning the radical of a module category.

Furthermore, the first author also  showed that such a bound has a close relationship with the length of the longest non-zero path from the projective in a vertex $a$ to the injective in the same vertex going through the simple in $a$.

Later, the first and second named authors reduced the steps in order to compute such a bound proving that if $A$ is a finite dimensional algebra over an algebraically closed field, then  it is not necessary to analyze the vertices of the ordinary quiver which are either sinks or sources, see \cite{CG2}.

Continuing with this line of research, the aim  of  this article is to determine the vertices of $Q_A$ that are sufficient to be consider in order to determine the nilpotency index of the radical of the module category of an algebra.

We say that a vertex $u$ in $Q_A$ is involved in  a zero-relation of an ideal $I_A$, whenever $\alpha_m \dots \alpha_1 \in I_A$  then  $u =s(\alpha_i)$ for some $i=2, \dots, m$.

Here, we solve the problem for some representation-finite algebras such that their Auslander-Reiten quiver are components with length and we also give a solution for string algebras. Precisely, in case that the Auslander-Reiten quiver is a component with length, we prove Theorem A and Theorem B.
\vspace{.05in}

{\bf Theorem A}. {\it
Let $A\simeq kQ_A/I_A$ be a representation-finite monomial algebra where the Auslander-Reiten quiver is a component with length. Consider $(R_A)_0$ to be the set of  the vertices $u\in Q_A$
involved in zero-relations of $I_A$. Then the nilpotency index of $\Re(\emph{mod}\,A)$ is determined by the length of the longest  path of irreducible morphisms from the projective $P_u$ to the injective $I_u$, where $u \in(R_A)_0$.  More precisely, it is enough to study only one vertex of each zero-relation of $A$.}
\vspace{.05in}

%{\color{red} More precisely, we prove that it is enough to study one vertex of each relation of $A$.}

{\bf Theorem B}. {\it Let $A \simeq kQ_A/I_A$ be a  representation-finite toupie algebra where $I_A$ has only  commutative relations. Then any vertex of $Q_A$ which is in the shortest  branch of $Q_A$, not being the sink or the source of $Q_A$,  is the one that  we may consider to determine the nilpotency index of the radical of $\emph{mod}\,A$.}
\vspace{.05in}

As an application of Theorem A, in Theorem \ref{nilpo} we present a formula to compute the nilpotency index of the radical of some representation-finite tree algebras with only one zero-relation. In Proposition \ref{ejemplos}, Proposition \ref{interDn} and Proposition \ref{E6} we present some tree algebras satisfying such a formula.
Furthermore, we also give a solution to find the bound  when the ordinary quiver of the algebra is, roughly speaking,  a sequence   of some particular tree algebras glued by   $m$ zero-relations not overlapped,  see Theorem \ref{m-relations}.

On the other hand, applying Theorem B, we find a formula depending on the quiver of  $Q_A$, to obtain the nilpotency index of the radical of a  representation-finite toupie algebra where $I_A$ has only  commutative relations, see Theorem \ref{three-commutative} and Theorem \ref{two-commutative}.
\vspace{.05in}

In \cite{CG},  the first and second authors showed how to read the minimal lower bound $m \geq 1$, such that the $m$-th power of
the radical of the module category of a string algebra $A$ vanishes,  in terms of strings of $Q_A$.
Here, we concentrate to study how to reduce the steps to compute the nilpotency index of the radical of the module category of a string algebra. Observe that for this algebras the Auslander-Reiten quiver may be a component without length. Precisely, we  prove Theorem C.
\vspace{.05in}

{\bf Theorem C}. {\it Let $A\simeq kQ_A/I_A$ be a string algebra of finite representation type. Consider $(R_A)_0$ to be the set of the vertices $u \in (Q_A)_0$
such that they are involved in zero-relations of $I_A$. Then the nilpotency index of $\Re(\emph{mod}\,A)$ is determined by the length of the longest non-zero path of irreducible morphisms between indecomposable modules from the projective $P_u$ to the injective $I_u$ going through the simple $S_u$, where $u\in(R_A)_0$.}
\vspace{.05in}

Pullbacks of finite dimensional algebras over an algebraically closed field and the theory of one-point extension and one-point coextension of algebras are  considered to prove the results, see for instance \cite{CW} and \cite{SS}.
\vspace{.05in}

This paper is organized as follows. In the first section, we state some notations and recall some preliminaries results.
Section 2 is dedicated to prove Theorem A, Theorem B, and Theorem C.
In Section 3, we apply the results of Section 2 to obtain formulas to compute the nilpotency index of the radical of the module category of some algebras.
\vspace{.1in}

\thanks{The authors thankfully acknowledge partial support from CONICET
and from Universidad Nacional de Mar del Plata, Argentina. The first author is a researcher from CONICET. The authors also thank Sonia Trepode for useful conversations.}

\section{preliminaries}

Throughout this work, by an algebra we mean a finite dimensional $k$-algebra over an algebraically closed field,  $k$.

\begin{sinob} %{\bf Notations}
%\vspace{.1in}

A {\it quiver} $Q$ is given by a set of vertices $Q_0$ and
a set of arrows $Q_1$, together with two
maps $s,e:Q_1\fle Q_0$. Given an arrow $\al\in Q_1$, we write $s(\al)$
the starting vertex of $\al$ and $e(\al)$
the ending vertex of $\al$. By $\overline{Q}$ we denote the underlying graph of $Q$.

Let $A$ be an algebra. There exists a quiver $Q_A$, called the {\it ordinary quiver of}  $A$, such that
$A$ is the quotient of the path algebra $kQ_A$ by an admissible ideal.

We denote by $\mbox{mod}\,A$ the category of finitely generated
right $A$-modules and  by  $\mbox{ind}\,A$ the full subcategory of $\mbox{mod}\,A$ which consists of
one representative of each isomorphism class of indecomposable $A$-modules.

We say that an algebra A is {\it representation-finite} if there is only a finite number
of isomorphisms classes of indecomposable A-modules.

We denote by $\Gaa$ the Auslander-Reiten
quiver of $\mbmA$, by $\tau$ and $\tau^{-1}$ the Auslander-Reiten translation $\mbox{DTr}$ and $\mbox{TrD}$, respectively.

For unexplained notions in representation theory of algebras, we refer the reader to \cite{ASS} and \cite{ARS}.
\end{sinob}

\begin{sinob} \label{rad}%{\bf On the radical}
%\vspace{.1in}

Consider  $X,Y \in \mbox{mod}\,A$. The {\it radical} of $\mbox{Hom}_A(X,Y)$, denoted by   $\Re(X,Y)$,  is
the set of all the morphisms $f: X \rightarrow Y$ such that, for each
$M \in \mbox{ind}\,A$, each $h:M \rightarrow X$ and each $h^{\prime }:Y\rightarrow M$
the composition $h^{\prime }fh$ is not an isomorphism. For $n \geq 2$, the powers of the radical
are inductively  defined  as $\Re^{n}(X,Y)= \Re^{n-1}(M,Y)\Re(X,M) = \Re (M,Y)\Re^{n-1}(X,M)$ for some $M \in \mbox{mod}\,A$.

A morphism $f :X  \rightarrow  Y$, with $X,Y \in \mbox{ mod}\,A$,
is called {\it irreducible} provided it does not split
and whenever $f = gh$, then either $h$ is a split monomorphism or $g$ is a
split epimorphism.

It is well-known that an algebra $A$ is representation-finite
if and only if there is a positive integer $n$ such that $\Re^n(X,Y)=0$ for each $A$-module $X$ and $Y$.
The minimal lower bound $m$ such that $\Re^m(\mbox{mod}\,A)=0$ is called the \textit{nilpotency index} of
$\Re(\mbox{mod}\,A)$. We denote it by $r_A$.

We denote by $P_a$, $I_a$ and  $S_a$ the projective, the injective and the simple $A$-module corresponding to the
vertex $a$ in $Q_A$, respectively. For all $a \in (Q_A)_0$, let $r_a$ be the length of the non-zero path of irreducible morphisms between indecomposable $A$-modules from $P_a$ to $I_a$ going through $S_a$.

In \cite{C}, the first named author found the nilpotency index of a representation-finite algebra in terms of  the length of the above mentioned particular paths. Precisely, the author proved the following result, fundamental for our purposes.

\noindent\begin{teo}\emph{({\cite{C}})} \label{cla} Let $A\simeq kQ_A/I_A$ be a finite
dimensional algebra over an algebraically closed field and assume
that $A$ is  representation-finite. Then, the nilpotency index of the radical of the module category of $\emph{mod}\,A$ is
$\;\emph{max}\, \{r_a+1\}_{a \in Q_0}$.
\end{teo}

As an immediate consequence of \cite[Proposition 3.11]{CG2} and its dual,
to compute the nilpotency index of the radical of a module
category of a representation-finite algebra, it is enough to analyze
the vertices which are neither sinks nor sources in $Q_A$, as we state below.

\begin{teo}\label{nilpopf} \cite[Theorem 3.16]{CG2}
Let $A=kQ_A/I_A$ be a representation-finite algebra.
Consider $\mathcal{V}$ to be the set of vertices of $Q_A$ which are neither sinks nor sources.
Suppose that $\mathcal{V}\neq \emptyset$. Then the  nilpotency index of $\Re(\emph{mod}\,A)$ is
$\emph{max}\,\{r_a+1\}_{a \in \mathcal{V}}$.
\end{teo}

Note that if $\mathcal{V}$ is empty, then the algebra is hereditary and we have the following result.

\begin{teo}\cite[Theorem 4.11]{Z} \label{ppalher}
Let $H=kQ$ be a representation-finite hereditary algebra and let $r_H$
be the nilpotency index of $\Re(\emph{mod}\,H).$ The following conditions hold.
\begin{enumerate}
\item[(a)] If $\overline{Q}=A_n$, then $r_H=n,$ for $n\geq 1$.
\item[(b)] If $\overline{Q}=D_n$, then $r_H=2n-3,$ for $n\geq 4$.
\item[(c)] If $\overline{Q}=E_6$, then $r_H=11.$
\item[(d)] If $\overline{Q}=E_7$, then $r_H=17.$
\item[(e)] If $\overline{Q}=E_8$, then $r_H=29.$
\end{enumerate}
\end{teo}
\end{sinob}

\begin{sinob} \label{orbit}%{\bf Components and the orbit graph}
%\vspace{.1in}

Let $\Gamma$ be a component of $\Gamma_A$. Following \cite{CPT}, we say that $\Gamma$ is a component {\it with length}
if parallel paths in $\Gamma$ have the same length. By {\it parallel paths} we mean paths of irreducible morphisms between indecomposable $A$-modules
having the same starting vertex and the same ending vertex.
By the arrow $\rightsquigarrow$  we denote  a  path of irreducible morphisms between indecomposable $A$-modules.

We say that a path $Y_1 \rightarrow Y_{2}\rightarrow
\dots\rightarrow Y_{n-1}\rightarrow Y_n$ in $\Gamma_A$ is {\it sectional} if for each $i=1, \dots, n-1$,
we have that $Y_{i+1} \not \simeq  {\tau }Y_{i-1}$,  see  \cite{IT}.

Given a directed component $\Gamma$ of $\Gamma_A$, its {\it orbit graph}, denoted by $O(\Gamma )$, has as points the $\tau$-orbits
$O(M)$ of the modules M in $\Gamma$. There exists an edge between $O(M)$ and $O(N)$ in $O(\Gamma )$
whenever there are integers $m, n$  and an irreducible morphism from $\tau^m M$ to $\tau^n N$ or from $\tau^n N$ to $\tau^m M$.

We recall that if the orbit graph  $O(\Gamma)$ is of tree-type,
then $\Gamma$ is a simply connected translation quiver, and therefore, by \cite{BG},
$\Gamma$ is a component with length.
\end{sinob}

\begin{sinob}%{\bf String algebras}
%\vspace{.1in}

Let $A$ be an algebra such that $A \simeq kQ_A/I_A$. The algebra $A$ is called a {\it string algebra} provided:

\begin{enumerate}
\item[(1)] Any vertex of $Q_A$ is the starting point of at most two arrows.
\item[(1')] Any vertex of $Q_A$ is the ending point of at most two arrows.
\item[(2)] Given an arrow $\beta$, there is at most one arrow $\gamma$ with $s(\beta)=e(\gamma)$ and $\beta\gamma \notin I_A.$
\item[(2')] Given an arrow $\gamma$, there is at most one arrow $\beta$ with $s(\beta)=e(\gamma)$ and $\beta\gamma \notin I_A.$
\item[(3)] The ideal $I_A$ is generated by a set of paths of $Q_A$.
\end{enumerate}
\vspace{.05in}

A {\it string} in $Q_A$ is either a trivial path
$\varepsilon_v$ with $v\in Q_0$, or a reduced walk $C=c_n\dots c_1$ of length $n\geq 1$ such that
no sub-walk $c_{i+t}\dots c_i$ nor its inverse belongs to $I_A$.
We say that a string $C=c_n\dots c_1$ is {\it direct (\it inverse)}
provided all $c_i$ are arrows (inverse of arrows, respectively). We
consider the trivial walk $\eps_v$ a direct as well as an inverse string.

For general theory on string algebras, we refer the reader to \cite{BR}.
\end{sinob}

\begin{sinob} %{\bf One-point extension and coextensions algebras}
%\vspace{.1in}

Let $A$ be an algebra and $X$ be an $A$-module.
The {\it one-point extension} of $A$ by $X$, which we denote by $A[X]$, is the 2x2 matrix algebra

 \[ A[X]= \left[ \begin{array}{cc}
A & 0 \\
  {_k}X_A & k
\end{array}
\right] \]

\noindent with the ordinary addition of matrices and the multiplication induced from the usual $k$-$A$-bimodule structure ${_k}X_A$ of $X$.

The {\it one-point coextension} of $A$ by $X$, which we denote by $[X]A$, is the 2x2 matrix algebra

\[ [X]A= \left[ \begin{array}{cc}
k & 0 \\
  DX & A
\end{array}
\right] \]

\noindent with the ordinary addition of matrices and the multiplication induced from the usual $A$-$k$-bimodule structure $DX =\mbox{Hom}_k({_k}X_A, k)$ induced by the $k$-$A$-bimodule structure of ${_k}X_A$.
\vspace{.1in}

We recall the following useful results from \cite[Chapter XV, Corollary 1.7]{SS} of one-point extension and one-point coextension algebras.

\begin{prop} \label{lema S-S} Let $A$ be an  algebra and $X$ be an $A$-module. If
$0 \rightarrow L \rightarrow M \rightarrow N \rightarrow 0$
is an almost split sequence in $\emph{mod}\,A$ such that $\emph{Hom}_A(X, L) = 0$, then this sequence remains almost split in $\emph{mod}\,A[X]$.
\end{prop}

A dual result holds for one-point coextension algebras.

\begin{prop}\label{dual lema S-S} Let $A$ be an  algebra and $Y$ be an $A$-module. If
$0 \rightarrow L \rightarrow M \rightarrow N \rightarrow 0$
is an almost split sequence in $\emph{mod}\,A$ such that $\emph{Hom}_A(N, Y) = 0$, then this sequence remains almost split in $\emph{mod}\,[Y]A$.
\end{prop}
\end{sinob}

\begin{sinob} %{\bf Linearly oriented pullback algebras}
%\vspace{.1in}

Let $Q$ be a full subquiver of the quivers $Q'$ and $Q''$ and $i:Q\rightarrow Q'$, $j: Q\rightarrow Q''$ be inclusions quivers.
The {\it pushout} of $i$ and $j$ is the quiver where the vertices are those of $Q'$ and those of $Q''$ which are not in $j(Q)$, and the arrows are those of $Q'$ and

\begin{enumerate}
\item[$\bullet$] for $x,y \in Q''_0 \backslash j(Q)_0$ there is one arrow from $x$ to $y$ for each arrow from $x$ to $y$ in $Q''$;
\item[$\bullet$] for $x \in Q''_0 \backslash j(Q)_0$ and $y=i(z)$ for some $z \in Q$,  there is one arrow from $x$ to $y$ for each arrow from $x$ to $j(z)$ in $Q''$;
\item[$\bullet$] for $x=i(z)$ for some $z\in Q$ and $y \in Q''_0 \backslash j(Q)_0$, there is one arrow from $x$ to $y$ for each arrow from $j(z)$ to $y$ in $Q''$.
\end{enumerate}

Following \cite{CW}, consider $C\simeq kQ_C/I_C$ and $D\simeq kQ_D/I_D$ to be algebras and $Q_B$ to be a full and convex subquiver of the quivers $Q_C$ and $Q_D$ such that $I_C\cap kQ_B=I_D\cap kQ_B$. Denote this ideal by $I_B$ and consider the algebra $B \simeq kQ_B/I_B$.
So $B \simeq e_B C e_B \simeq  e_B D e_B$ is a common quotient of $C$ and $D$. Let $f_{C}: C \twoheadrightarrow B$ and $f_{D}: D \twoheadrightarrow B$ be the natural surjective morphisms given respectively, by $c \mapsto e_{_B} c e_{_B}$ and $d \mapsto e_{_B} d e_{_B}$. The {\it pullback} of $f_{C}$ and $f_{D}$ is the algebra $A=\{(c, d)\in C \times D \mid  e_{_B} c e_{_B} = e_{_B} d e_{_B} \}$.

By \cite[Theorem 1]{CW}, if  $A$ is the pullback of two morphisms of algebras $f_{C}: C \rightarrow B$ and $f_{D}: D\rightarrow B$, $Q_A$ is the pushout of the inclusion of quivers $Q_B  \rightarrow Q_C$ and $Q_B \rightarrow Q_D$ and $I_A$ is the ideal of $kQ_A$ generated by $I_C$, $I_D$ and the paths linking $(Q_C)_0 \backslash (Q_B)_0$ and $(Q_D)_0 \backslash (Q_B)_0$ then $A \simeq kQ_A/I_A$.

We say that $A$ is  a {\it linearly oriented pullback} of $C \fepi B$ and  $D \fepi B$ if there is no path from $(Q_B)_0$ to
$(Q_C)_0 \backslash (Q_B)_0$ and neither from $(Q_D)_0 \backslash (Q_B)_0$ to $(Q_B)_0$.

Given $M, N $ in $\mbox{ind}\,A$, a path from $M$ to $N$ in  $\mbox{ind}\,A$ is a sequence of non-zero morphisms $M=X_1 \rightarrow \dots \rightarrow X_t=N$ with $t \geq 1$, where $X_i$ are indecomposable for all $i=1, \dots, t$. We say that $M$ is a {\it predecessor} of $N$ and that $N$ is a {\it successor} of $M$.

For a module $M$, we denote by $\mbox{Succ}\,M$ the full subcategory of $\mbox{ind}\;A$ consisting of all the successors of any indecomposable direct summand of $M$ and by $\mbox{Pred}\,M$ the full subcategory of $\mbox{ind}\,A$ consisting of all the predecessors of any indecomposable direct summand of $M$.
\end{sinob}

\section{On the nilpotency index}

Let $A$ be an algebra. Throughout this work, we denote by  $r_A$ the nilpotency index of the radical of the module category of $A$.

The aim of this section is to find the vertices $a \in (Q_A)_0$ where the paths from the projective $P_a$ to the injective $I_a$ going through $S_a$ are of maximal length.

\subsection{Components with length}
First, we study the above mentioned problem for  algebras such that their Auslander-Reiten quiver is a component with length.

\begin{defi}\label{involved}
Let $A \simeq kQ_A/I_A$.  We say that a vertex $x$ in $Q_A$ is involved in a zero-relation $\alpha_m \dots \alpha_1 \in I_A$ if $x=s(\alpha_i)$ for some $i=2, \dots, m$.
\end{defi}

We denote by $(R_A)_0$ the set of all the vertices $u \in (Q_A)_0$ involved in zero-relations of $I_A$.

\begin{defi}\label{overlap}
Let $A \simeq kQ_A/I_A$ be an algebra and $I_A=< \chi >$,  where $\chi$ is a minimal set of relations. Let $w$ be a path in $k Q_A$ and  $\chi'=\chi \mid_{w}$. Two paths $\gamma_1, \gamma_2\in \chi'$ are said {\it overlap} each other if $s(\gamma_1)<s(\gamma_2)<e(\gamma_1)$ or if $s(\gamma_2)<s(\gamma_1)<e(\gamma_2)$.
\end{defi}

We start proving some lemmas.

\begin{lema}\label{zero-vertex}
Let $A \simeq kQ_A/I_A$ be a representation-finite  monomial algebra with $r$ zero-relations not overlapped, and  $r\geq 2$. Let $\Gamma_A$ be a component with length. The following conditions hold.
\begin{enumerate}
\item[(a)] If $a$ and $b$ are vertices of $(R_A)_0$ involved in the same zero-relation of $I_A$ then $r_a= r_b$.
\item[(b)] If $a$ and $b$ are vertices in $(Q_A)_0\backslash (R_A)_0$ and there is an arrow from $a$ to $b$, then
 $r_a= r_b$.
\item[(c)]  If $a$ is a vertex in $(R_A)_0$ and $b$ is a vertex in $(Q_A)_0\backslash (R_A)_0$,  then $r_a\geq r_b$.
\end{enumerate}
\end{lema}

\begin{proof}
(a) Let $a,b$  be two vertices involved in the same zero-relation of $I_A$. Without loss of generality, we may assume that there is an arrow from $a$ to $b$ in $Q_A$. Hence any morphism from $P_b$ to $P_a$ is irreducible, since  $P_b$ is a direct summand of $\mbox{rad}\, P_a$.

On the other hand, a morphism from $I_b$ to $I_a$ is irreducible too.
%In fact, either $I_a =I_b /\mbox{soc}\,I_b$ or $I_a$ is a direct summand of $I_b /\mbox{soc}\,I_b$.
Therefore, we have the following  situation in $\Gamma_A$

\begin{equation} \label{path}
 \xymatrix @!0 @R=0.4cm  @C=1.6cm  { & P_a \ar@{~>}[rr]  &  & S_a \ar@{~>}[rdd]   &   \\
 &&&& \\
P_b \ar[ruu] \ar@{~>}[rdd] & &   &  & I_a  \\
      & & &  &   \\
      & S_b  \ar@{~>}[rr]  & &{I_b} \ar@{~>}[ruu]^{\gamma}   &\\}
 \end{equation}

\noindent where the length of the path $\gamma$ from ${I_b}$ to ${I_a}$ is one.

Since the Auslander-Reiten quiver of $\mbox{mod}\, A$ is a component with length,  then the length of any path from $P_a$ to $I_a$ is equal to the length of any path from $P_b$ to $I_b$, proving the result.

(b) With similar arguments than in the proof of the above statement, we get the result.

(c) Let $a \in (R_A)_0$ and $b \in (Q_A)_0 \backslash (R_A)_0$. Again, without loss of generality,  we may assume that there is an arrow from $a$ to $b$ in $Q_A$. Then any morphism from $P_a$ to $P_b$ is  irreducible. Moreover, there is a non-zero morphism from  $I_a$ to $I_b$ of length greater than or equal to one. Precisely, if $\alpha_m \dots \alpha_1$ is the zero-relation in $I_A$ and $b$ is such that $e(\alpha_m)=b$  then the length  of any path from  $P_a$ to $I_a$  is greater than the length of any path from  $P_b$ to $I_b$. In any other case, the length is equal.

Note that we have paths in  $\Gamma_A$ as in $(\ref{path})$, where the length of $\gamma$ is equal to $s$, with $s \geq 1$.
Assume that the length of  the paths from  $P_a$ to $I_a$ going through $S_a$ is $n$ and that the length of the paths from  $P_b$ to $I_b$ going through $S_b$ is $l$. Since the Auslander-Reiten quiver is a component  with  length,  then $1+n=l+ s$. Therefore $n \geq l$.
\end{proof}

In case there is an overlapped relation it is sufficient to analyze the vertices that belong to the intersection of both relations as we prove below.

\begin{lema}\label{nonzero-vertex}
Let $A \simeq kQ_A/I_A$ be a monomial algebra where $I_A$ has two paths overlap  each other and where $\Gamma_A$ is a component with length. The following conditions hold.
\begin{enumerate}
\item[(a)] If $a$ and $b$ are vertices in $(R_A)_0$ involved in the intersection of the two overlapped paths, then the length of any path from $P_a$ to $I_a$ is equal to the length of any path from $P_b$ to $I_b$.
\item[(b)] If $a$ is a vertex in $(R_A)_0$ involved in the intersection of the two overlapped paths, then the length of the paths from $P_a$ to $I_a$ are greater than or equal to the length of the paths from $P_b$ to $I_b$, where $b$ is a vertex not involved in the intersection of the two overlapped paths.
\end{enumerate}
\end{lema}

\begin{proof}
(a) Consider the intersection of the two paths overlap each other. Let $a$  and $b$ be vertices involved in such intersection. With similar arguments that in the proof of Lemma \ref{zero-vertex} (a),  we get the result.

(b) Let $a$  be a vertex involved in the intersection of the overlapped  zero-relations, let say $a \in \gamma_1 \cap \gamma_2$. First, consider the vertex $a \in \gamma_1$. With a similar proof that in Lemma \ref{zero-vertex} (b), for any vertex  $b$ not involved in $\gamma_1$  we have that the length of the paths from $P_a$ to $I_a$ are greater than or equal to the length of any path from $P_b$ to $I_b$.

On the other hand, if we consider $a$ in $\gamma_2$ then again with a similar proof that in Lemma \ref{zero-vertex} (b), for any vertex  $b$ not involved in $\gamma_2$  we have that the length of the paths from $P_a$ to $I_a$ are greater than or equal to the length of any path from $P_b$ to $I_b$.
\end{proof}

Let $\mathcal{S} \subset (R_A)_0$ be defined as follows:
\begin{enumerate}
\item[$\bullet$] for each two paths $\gamma_1$ and $\gamma_2$ in $I_A$ overlap each other, we chose a vertex $a$
involved in $\gamma_1 \cap \gamma_2$, and
\item [$\bullet$]  for each different not overlapped zero-relation $\rho$ in $I_A$, we chose a vertex $b$ involved in $\rho$.
\end{enumerate}

By Theorem \ref{cla} and the above lemmas we precise the result state in Theorem A.

\begin{teo}\label{vertex-nilpo}
Let $A \simeq kQ_A/I_A$ be a representation-finite monomial algebra. Assume that $\Gamma_A$ is a component with length. The vertices of $Q_A$ that we have to consider in order to determine the nilpotency index of the radical of $\emph{mod}\,A$ are the ones involved in the zero-relations of $I_A$.
Precisely,  $r_A = \emph{max}\, \{ r_{a}+1 \}_{a \in \mathcal{S}}$.
\end{teo}

As an immediate consequence of Theorem \ref{vertex-nilpo} and the fact that the Auslander-Reiten quiver of a representation-finite monomial tree algebra is  with length, we have the following corollary.

\begin{coro}\label{coro-vertex}
Let $A \simeq kQ_A/I_A$ be a representation-finite algebra of tree type. The vertices that we have to  consider in order to determine the nilpotency index of the radical of the module category of $A$ are the ones involved in $\mathcal{S} \subset (R_A)_0$.
\end{coro}

\begin{ej}
\emph{Consider the algebra $A=Q_A/I_A $ given by the presentation:
}
\vspace{.1in}

$${\xymatrix   @R=0.4cm  @C=0.6cm {
1\ar@{--}@/^{5mm}/[rrrr]\ar[r]_{\alpha_1}&2\ar@{--}@/^{5mm}/[rrrr]\ar[r]_{\alpha_2}&3\ar[r]_{\alpha_3}&4\ar[r]_{\alpha_4}
&5\ar[r]_{\alpha_5}&6\ar@{--}@/^{5mm}/[rrr]&7\ar[l]^{\beta_1}&8\ar[l]^{\beta_2}&9\ar[l]^{\beta_3}&10\ar[l]
%\ar@{--}@/_{5mm}/[rrrrr]\ar[r]^{\alpha}&2\ar[r]^{\beta}&3\ar[r]^{\gamma}&5\ar[r]^{\delta}&6\ar[r]^{\lambda}&7&&
}}$$

\noindent \emph{where $I_A=<\gamma_1,\gamma_2,\rho>$, with $\gamma_1=\alpha_4\alpha_3\alpha_2\alpha_1$, $\gamma_2=\alpha_5\alpha_4\alpha_3\alpha_2$ and $\rho=\beta_3\beta_2 \beta_1$.}

\emph{We can observe that $\gamma_1$ and $\gamma_2$ are two path of $I_A$ overlap each other,
and the vertices $3$ and $4$ are the ones that are involved in $\gamma_1\cap \gamma_2$.
On the other hand, $\rho$ is the unique zero-relation that is not overlapped, and
the vertices $7$ and $8$ are the ones that are involved in it. Then we can
choose, for instance, the set $\mathcal{S}=\{3,7\},$ and following Theorem \ref{vertex-nilpo}
we have that
$$r_A=\mbox{max},\{r_3+1,r_7+1\}.$$}
\end{ej}

\subsection{On the nilpotency index of a toupie algebra with commutative relations}
A finite non linear quiver $Q$ is called a {\it toupie} if it has a unique source, a
unique sink, and for any other vertex $x$ in $Q$ there is exactly one arrow starting
in $x$ and exactly one arrow ending in $x$.

In case  $A \simeq kQ_A/I_A$  is  a toupie algebra with only commutative relations, then it is simply connected and therefore the Auslander-Reiten quiver is a component with length. In addition, if  $A$ is representation-finite then it at most three branches. We shall prove that in this case the nilpotency index of the radical of their module category is given by the length of the path from $P_x$ to $I_x$ where $x$ is a vertex in a branch of  $Q_A$ that has the shortest length among all branches of the given toupie. Precisely, we prove the following result.

\begin{teo}\label{vertex-commutative}
Let $A \simeq kQ_A/I_A$ be the representation-finite algebra

$$
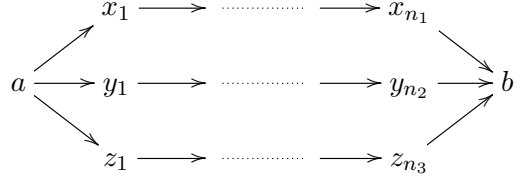
\xymatrix @!0 @R=1.0cm  @C=1.3cm {&x_1\ar[r]&\ar@{.}[r]&\ar[r]&x_{n_1} \ar[rd]& \\
a\ar[r]\ar[ru]\ar[rd]&y_1\ar[r]&\ar@{.}[r]&
\ar[r]&y_{n_2}\ar[r]&b\\
&z_1\ar[r]&\ar@{.}[r]&\ar[r]&z_{n_3}\ar[ru]&\\} $$

\noindent with commutative relations. Then any vertex of $Q_A$ which is in the  branch that has the shortest length among all branches of $Q_A$, and that is not the sink or the source of $Q_A$,  is the one that  we have to consider in order to determine the nilpotency index of the radical of $\emph{mod}\,A$.\end{teo}

\begin{proof} Consider all the vertices in a branch of $Q_A$, which are not the sink or the source of $Q_A$. Without loss of generality, we may consider the vertices $x_1, \dots, x_{{n_1}}$. Then $r_{x_i}=r_{x_j}$,
%$\ell(P_{x_{i}} \rightsquigarrow I_{x_{i}}) = \ell(P_{x_{j}} \rightsquigarrow I_{x_{j}})$,
for $i \neq j$ and $1 \geq i,j \leq n_1$. In fact, assume that $i <j$. Then there are paths of irreducible morphisms in $\Gamma_A$ as follows:
$$P_{x_{j}} \stackrel{\gamma} \rightsquigarrow I_{x_{j}} \rightarrow I_{x_{j-1}} \rightarrow \dots \rightarrow I_{x_{i}}$$
\noindent and
$$P_{x_{j}} \rightarrow P_{x_{j-1}} \rightarrow \dots \rightarrow  P_{x_{i}} \stackrel{\gamma'}\rightsquigarrow I_{x_{i}},$$
\noindent where $I_{x_{j}} \rightarrow I_{x_{j-1}} \rightarrow \dots \rightarrow I_{x_{i}}$ and $P_{x_{j}} \rightarrow P_{x_{j-1}} \rightarrow \dots \rightarrow  P_{x_{i}}$ are  paths of irreducible morphisms between indecomposable injective and projective $A$-modules, respectively, of the same length.
Since $\Gamma_A$ is a component with length then clearly $\ell({\gamma})  =\ell({\gamma'})$.

On the other hand, there are paths of irreducible morphisms between indecomposable modules in $\Gamma_A$  as follows:

$$P_b \rightarrow P_{x_{n_1}}\rightarrow \dots \rightarrow P_{x_1} \rightsquigarrow I_{x_1}\rightarrow I_a,$$ $$P_b \rightarrow P_{y_{n_2}}\rightarrow \dots \rightarrow P_{y_1} \rightsquigarrow I_{y_1}\rightarrow I_a,$$ and $$P_b \rightarrow P_{z_{n_3}}\rightarrow \dots \rightarrow P_{z_1} \rightsquigarrow I_{z_1}\rightarrow I_a.$$

Since $\Gamma_A$ is a component with length, then  the path of longest length from $P_{w_{1}}$ to $I_{w_{1}}$ with $w \in \{x,y,z \}$ is the path whose vertex belongs to the shortest branch of $Q_A$, proving the result.
\end{proof}

\subsection{On the nilpotency index of a string algebra} In this section, we prove that it is enough to consider the vertices involved in the zero-relations to determine the nilpotency index of $\Re(\mbox{mod}\, A)$ whenever $A$ is a representation-finite string algebra. Note that in this algebras, the Auslander-Reiten quiver is not necessarily a component with length.
\vspace{.05in}

Let $A=kQ_A/I_A$ be a representation-finite string algebra. Following \cite{CG}, for each $u \in (Q_{A})_0$ we consider the  quivers $Q_u^e$ and $Q_u^s$ defined as follow:
\benu
\item[(a)]
\begin{enumerate}
	\item The vertices of $(Q_u^e)_0$ are the strings $C$ in $Q_A$ such that $e(C)=u,$ where $C$ is either the trivial walk $\eps_u$ or $C=\alpha C'$, with $\al \in (Q_{A})_1$.
	\item If $a=C$ and $b=C'$ are two vertices of $(Q_u^e)_0$, then there is an arrow
	$a\fle b$ in $Q_u^e$ if $C'$ is the reduced walk of $C\be^{-1}$, for some $\beta \in (Q_{A})_1$.
\end{enumerate}
\item[(b)]
\begin{enumerate}
	\item The vertices of $(Q_u^s)_0$ are the strings $C$  in $Q_A$ such that $s(C)=u$, where $C$ is either the trivial walk $\eps_u$ or $C=C'\alpha $, with $\alpha \in (Q_{A})_1$.
	\item If $a=C$ and $b=C'$ are two vertices of $(Q_u^s)_0,$ then there is an arrow
	$a \rightarrow b$ in $Q_u^s$ if $C'$ is the reduced walk of $\beta C$, for some $\beta \in (Q_{A})_1$.
\end{enumerate}
\enu

We recall from \cite[Proposition 3.2]{CG}  that $r_u=|(Q_u^s)_0|+|(Q_u^e)_0|-2$,  where $|\chi|$ denotes the cardinal of the set $\chi$.

Now, we are in position to prove the following result.

\begin{prop}
Let $A=kQ_A/I_A$ be a representation-finite string algebra with $I_A \neq 0$. Assume $x$ is a vertex in $Q_A$ which is neither a sink nor a source. Moreover, assume that $x$ is not involved in a zero-relation of $I_A$.  Then, there is $y\in (Q_{A})_0$ involved in a zero-relation of $I_A$ such that $r_x\leq r_y$.
\end{prop}
\begin{proof}
Consider  $x\in (Q_{A})_0$ a vertex which is neither a sink nor a source,  and such that  $x$  is not involved in a zero-relation of $I_A$. Since $Q_A$ is connected then there is a vertex $y\in (Q_{A})_0$  which is involved in a zero-relation of $I_A$ and moreover there is a string $C$ from $y$ to $x$ where none of the vertices of $C$ are involved in a zero-relation of $I_A$ (except for $y$). Note that in the vertex $x$ starts and ends exactly one arrow,  and moreover, in all the vertices of $C$ (except for $y$) there are at most
two edges.

Suppose that $C$ starts in an arrow. Then $C \in (Q_y^s)_0$. We have two cases to analyze; if either  $C$ ends in an arrow or if $C$ ends in the inverse of an arrow.

First, assume that $C$ ends in an arrow. Let $D \in (Q_x^s)_0$. Then $DC$ is a string, because $x$ is not involved in a zero-relation of $I_A$ and $DC \in (Q_y^s)_0$.
Since the string $DC\neq \eps_y$, then $\mbox{card}((Q_x^s)_0)\leq \mbox{card}((Q_y^s)_0)-1$.

Now, let $E\in (Q_x^e)_0$ with $\eps_x\neq E$. Then $E$ and $C$ end in an arrow. Hence $E^{-1}C$ is not a string, but we can consider the string obtained of $E^{-1}C$. Precisely, we have two possible situations:
\begin{enumerate}
\item $C=EC'$, with $C'$ a string, or
\item $E=CE'$, with $E'$ a string such that $E'\neq \eps_y$.
\end{enumerate}
If (1), then $C'$ is a string that belongs to $(Q_y^s)_0$. If (2), then ${E'}$ is a string. In case $E'$ ends in an arrow then $E'\in (Q_y^e)_0$. Otherwise, if $E'$ ends in the inverse of an arrow then $(E')^{-1}$ starts in an arrow $\delta$ with $s(\delta)=y$ and therefore $(E')^{-1}\in (Q_y^s)_0$.

Let $(Q_x^e)_1=\{E\,\mid\, C= EC'\}$, $(Q_x^e)_2=\{E\,\mid \, E=CE' \, \, \mbox{where} \, E' \,\,\mbox{ends in an arrow}\}$ and $(Q_x^e)_3=\{E\,\mid \, E=CE' \, \,\mbox{where} \, E' \,\, \mbox{ends in the inverse of an arrow}\}$. Then we can write $Q_x^e$ as the disjoint union
\[(Q_x^e)_0=  (Q_x^e)_1\sqcup (Q_x^e)_2 \sqcup (Q_x^e)_3.  \]

Since the strings in $Q_y^s$ that come from strings in $Q_x^s$ are different from the strings in $Q_y^s$ that come from strings in $Q_x^e-\{\eps_x\}$, then

%Since the strings in $Q_y^s$ that are also strings in $Q_x^s$ are different from the strings that belong to $Q_y^s$ and come from $Q_x^e-\{\eps_x\}$, then
\[|(Q_x^s)_0|+|(Q_x^e)_1|-1+|(Q_x^e)_3|\leq |(Q_y^s)_0|\]
\[|(Q_x^e)_2|\leq |(Q_y^e)_0|-1 .\]
\noindent Hence,
\[|(Q_x^s)_0|+|(Q_x^e)_0|\leq |(Q_y^s)_0|+|(Q_y^e)_0|\]
\noindent  and therefore, $r_x\leq r_y$.

Secondly, assume that $C$ ends in the inverse of an arrow. Let $E\in (Q_x^e)_0$. Then $E$ ends in an arrow and $e(E)=x$. Since $x$ is not involved in a zero-relation then $E^{-1}C$ is a string in $(Q_y^s)_0$.
Consider $D\in (Q_x^s)_0-\{\eps_x\}$. Since $x$ is neither a sink nor a source and it is not involved in a zero-relation of $I_A$, again we have two cases to analyze:
\begin{enumerate}
  \item[(3)] $C=D^{-1}C'$ or,
  \item[(4)] $D=D'C^{-1}$.
\end{enumerate}
If (3), then $C'$ is a string that starts in an arrow $\delta$ with $s(\delta)=y$. Then $C'\in (Q_y^s)_0$. If (4), then $D'$ is a string that starts in $y$. In case $D'$ starts in an arrow then $D'\in (Q_y^s)_0$. Otherwise, if $D'$ starts in the inverse of an arrow, then $(D')^{-1}\in (Q_y^e)_0$.

Let $(Q_x^s)_1=\{D\,\mid\, C=D^{-1}C'\}$, $(Q_x^s)_2=\{D\,\mid \, D=D'C^{-1} \; \mbox{where}\; D' \; \mbox{starts in an arrow}\}$ and  $(Q_x^s)_3=\{D\,\mid \, D=D'C^{-1} \; \mbox{where}\; D' \; \mbox{starts in the inverse of an arrow}\}$, again we can write $(Q_x^s)_0$ as the disjoint union

\[(Q_x^s)_0=(Q_x^s)_1\sqcup (Q_x^s)_2\sqcup(Q_x^s)_3.\]

Since the strings in $(Q_y^s)_0$ that come from strings in $(Q_x^e)_0$ are different from the strings in $(Q_y^s)_0$ that come from strings in $(Q_x^s)_0-\{\eps_x\}$, then

\[|(Q_x^e)_0|+|(Q_x^s)_1|-1+|(Q_x^s)_2|\leq |(Q_y^s)_0|\]
\[|(Q_x^s)_3|\leq |(Q_y^e)_0|-1.\]

\noindent Hence, $r_x\leq r_y$.

Finally, if $C$ starts in the inverse of an arrow, with a similar analysis as above we get that $r_x\leq r_y$.
\end{proof}

Let $A\simeq kQ_A/I_A$ be a  representation-finite string algebra. For
$\gamma\in (Q_A)_1$, we define the sets $\mathcal{S}_{\gamma}$ and $\mathcal{E}_{\gamma}$ whose elements are the strings
starting and ending in the arrow $\gamma$, respectively. Precisely,

$$\mathcal{S}_{\gamma}=\{C \text{ string in }Q_A\mid C=C'\gamma, \text{ with } C' \text{ a string in } Q_A \}$$

\noindent and

$$\mathcal{E}_{\gamma}=\{C \text{ string in } Q_A\mid C=\gamma C', \text{ with } C' \text{ a string in } Q_A \}. $$

\begin{prop}\label{relstring}
Let $A\simeq kQ_A/I_A$ be a  representation-finite string algebra and $\rho=\alpha_m\dots \alpha_1$ be a
non-overlapped zero-relation
in $I_A$, with $r\geq 3$. Let $x, y \in (Q_A)_0$ be two vertices involved in $\rho$ such that there
is an arrow $\alpha_i$ from $x$ to $y$, with $2\leq i\leq m-1$. Then, $r_x+|\mathcal{S}_{\delta}|=r_y+|\mathcal{E}_{\beta}|$,
where $\beta,\delta\in (Q_A)_1$ are  such that $e(\beta)=x$ with $\beta\neq \alpha_{i-1}$
and $s(\delta)=y$ with $\delta\neq  \alpha_{i+1}$.
\end{prop}

\begin{proof}
Consider  $\rho, x$ and $y$  as in the statement of the proposition.
Suppose  that there exist arrows $\beta_1: w_1\rightarrow x,$ $\beta_2:x\rightarrow w_2,$
 $\delta_1: z_1\rightarrow y$ and $\delta_2:y\rightarrow z_2$ different from the arrows $\alpha_j$ of $\rho$.
For the convenience of the reader, we illustrate the situation as follows:

$$
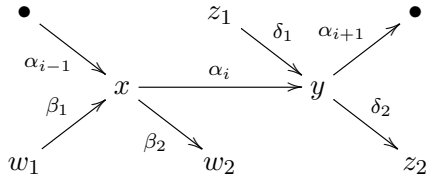
\xymatrix @!0 @R=1.0cm  @C=1.3cm {
 \bullet\ar[rd]_{ \alpha_{i-1}}&& z_1 \ar[rd]^{ \delta_1}&&\bullet\\
 &x\ar[rr]^{\alpha_i}\ar[rd]_{\beta_2}&&y\ar[ru]^{\alpha_{i+1}}\ar[rd]^{\delta_2}&\\
 w_1\ar[ru]^{ \beta_1}& &w_2 &&z_2} $$

Since $A$ is a string algebra, $\alpha_i\alpha_{i-1}$ and  $\alpha_{i+1}\alpha_{i}$ do not belong to $I_A$. Hence
we infer that $\alpha_i\beta_1$, $\beta_2\alpha_{i-1}$, $\alpha_{i+1}\delta_1$, and $\delta_2\alpha_{i}$ are in $I_A$.

Observe that $\mathcal{S}_{\alpha_i}=\{\alpha_i\}\cup\{C'\alpha_{i+1}\alpha_i\mid C' \text{ a string in } Q_A\}\cup
\{C''\delta_1^{-1}\alpha_i\mid C'' \text{ a string in } Q_A\}$. Hence, there
is a bijection between $\mathcal{S}_{\alpha_i}$ and $\{\eps_y\}\cup\mathcal{S}_{\alpha_{i+1}}\cup\mathcal{E}_{\gamma_1}$.

Similarly, we can establish a bijection between the sets $\mathcal{E}_{{\alpha}_i}$ and
$\{\eps_x\}\cup\mathcal{E}_{\alpha_{i-1}}\cup\mathcal{S}_{\beta_2}$.

By \cite[Proposition 3.2]{CG}, given $a\in (Q_A)_0$ we have that $r_a=|(Q^s_a)_0|+|(Q^e_a)_0|-2$.
Moreover, we have the following equalities of sets:
$(Q^s_x)_0=\{\eps_x\}\cup \mathcal{S}_{\alpha_i}\cup \mathcal{S}_{\beta_2}$,
$(Q^e_x)_0=\{\eps_x\}\cup \mathcal{E}_{\alpha_{i-1}}\cup \mathcal{E}_{\beta_1,}$
$(Q^s_y)_0=\{\eps_y\}\cup \mathcal{S}_{\alpha_{i+1}}\cup \mathcal{S}_{\gamma_2}$ and
$(Q^e_y)_0=\{\eps_y\}\cup \mathcal{E}_{\alpha_i}\cup \mathcal{E}_{\gamma_1}.$

Therefore, by the above bijections we deduce that $r_x+|\mathcal{S}_{\delta_2}|=r_y+|\mathcal{E}_{\beta_1}|,$
proving the result.
\end{proof}

The next corollary is an immediate consequence of Proposition \ref{relstring}.

\begin{coro}
With the above notations, let $\rho$ be the zero-relation $a\overset{\alpha_1}{\rightarrow }a_1\overset{\alpha_2}{\rightarrow }\dots {\rightarrow }a_{m-1}\overset{\alpha_m}{\rightarrow }b$. The following conditions hold.
\begin{enumerate}
\item[(i)] If  the vertices $a_i$, for $1\leq i\leq m-1$
    are  not involved in another relation of $I_A$, then $r_{a_i}=r_{a_j}$ for all $1\leq i,j\leq m-1.$
\item[(ii)] If $a_k$, with $1\leq k\leq m-1$, is a vertex which is involved in another relation of $I_A$, then
$r_{a_i}\leq r_{a_k}$ for all vertices  $a_{i}$ which are not involved in another relation different from $\rho$.
\end{enumerate}
\end{coro}

By the above results we get Theorem C.

\begin{teo} Let $A$ be a representation-finite string algebra.  Then the nilpotency index of $\Re(\emph{mod}\,A)$ is determined by the length of the longest non-zero path of irreducible morphisms between indecomposable $A$-modules from the projective $P_u$ to the injective $I_u$ going through the simple $S_u$, where $u\in(R_A)_0$.
\end{teo}

\section{Applications}

\subsection{Application to toupie  algebras with commutative relations}
Let  $A \simeq kQ_A/I_A$ be a toupie algebra with only commutative relations and with three branches. Then $A$ is a representation-finite algebra if and only if  at least one $n_i=1$ for $i=1, 2, 3$.

\begin{teo}\label{three-commutative}
Let $A \simeq kQ_A/I_A$ be the representation-finite algebra with three branches

$$
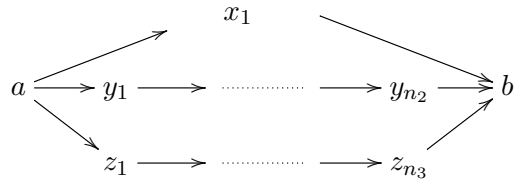
\xymatrix @!0 @R=1.0cm  @C=1.3cm {& & \; \; \;  \; \; \; x_1 & \ar[rrd]& & \\
a\ar[r]\ar[rru]\ar[rd]&y_1\ar[r]&\ar@{.}[r]&
\ar[r]&y_{n_2}\ar[r]&b\\
&z_1\ar[r]&\ar@{.}[r]&\ar[r]&z_{n_3}\ar[ru]& \\} $$

\noindent and  commutative relations. Let $B_1$ be the hereditary algebra obtained from $Q_A$ by deleting the source vertex and $B_2$ be the hereditary algebra obtained from $Q_A$ by deleting the sink vertex.  Then $r_A= 2 r_{B_1} -1= 2 r_{B_2} -1$.
\end{teo}
\begin{proof} By Theorem \ref{vertex-commutative}, we know that it is enough to study the length of a path from $P_{x_1}$ to $I_{x_1}$. We may consider the one  going through the simple  $S_{x_1}$. Consider $M= \mbox{rad} \, P_a$. Then $M= I_b$ in $\mbox{mod} \,B_1$ and $S_{x_1}$ is a direct summand of $I_b/ \mbox{soc}I_b$.
Hence there is an irreducible morphism from $I_b$ to $S_{x_1}$ in $\mbox{mod} \,B_1$. Moreover, $I_{x_1} =S_{x_1}$.

On the other hand, note that $A=B_1[I_b]$. By Proposition \ref{lema S-S}, since $B_1$ is directed and there is a path of irreducible morphisms from $P_{x_1}$ to $I_{b}$ in $\mbox{mod} \, B_1$ then we infer that such a path is also a path of irreducible morphisms in $\mbox{mod} \, A$. Furthermore, since
there is an irreducible morphism from $I_b$ to $S_{x_1}$ in $\mbox{mod}\,{B_1}$ then  $\mbox{Hom}_{B_1} (I_{b}, \tau S_{x_1})= 0$. Again, by Proposition \ref{lema S-S}, we know that the morphism from $I_{b}$ to $S_{x_1}$ is  irreducible in
$\mbox{mod} \,A$. Hence, $\ell_{B_1}(P_{x_1} \rightsquigarrow S_{x_1}) = \ell_{A}(P_{x_1} \rightsquigarrow S_{x_1})$ and since $I_{x_1} =S_{x_1}$ in $\mbox{mod}\,{B_1}$ then $r_{B_1}=\ell_{A}(P_{x_1} \rightsquigarrow S_{x_1})$.

Dually, with a similar argument as above and since $A= [P_a]B_2$, %where $A_2$ is the hereditary algebra obtained from $Q_A$ by deleting the vertex $b$,
we get that  $\ell_{B_2}(S_{x_1} \rightsquigarrow I_{x_1}) = \ell_{A}(S_{x_1} \rightsquigarrow I_{x_1})$. Therefore $r_{B_2}=\ell_{A}(S_{x_1} \rightsquigarrow I_{x_1})$ since $P_{x_1} =S_{x_1}$ in $\mbox{mod}\,{B_2}$. Then $r_A= r_{B_1}+ r_{B_2}-1$. Since $B_1$ and $B_2$ are of the same Dynkin type then $r_{B_1}=  r_{B_2}$ and  we get the result.
\end{proof}

Using similar arguments that in the proof of Theorem \ref{three-commutative}, we get the following consequence.

\begin{coro}\label{two-1-commutative}
Let $A$ be a toupie algebra of finite representation type with two branches and with a commutative relation. With the above notations, assume that $n_1=1$ or $n_2=1$.  Then $r_A= 2 r_{B_1}-1$.\end{coro}

We observe that if $n_1$ or $n_2$ are not equal to one in Corollary \ref{two-1-commutative}, then the formula does not hold. In case
$A$ is a representation-finite algebra with two branches and a  commutative relation we have the following more general result.

\begin{teo}\label{two-commutative}
Let $A \simeq kQ_A/I_A$ be the representation-finite algebra with two branches

$$
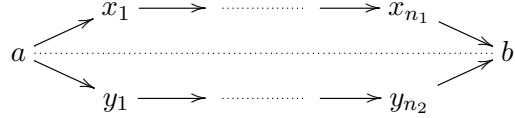
\xymatrix @!0 @R=0.6cm  @C=1.3cm {&x_1\ar[r]&\ar@{.}[r]&\ar[r]&x_{n_1} \ar[rd]& \\
a\ar[ru] \ar[rd]\ar@{.}[rrrrr]&&&&&b \\
&y_1\ar[r]&\ar@{.}[r]&
\ar[r]&y_{n_2}\ar[ru]&
} $$

\noindent and a commutative relation. Assume that $n_2 \geq n_1$.  Then $r_A= n_1 + 2 n_2 +2 $. \end{teo}
\begin{proof}
By Theorem \ref{vertex-commutative}, it is enough to study the length of any path from $P_{x_1}$ to $I_{x_1}$. We may consider the one going through the simple  $S_{x_1}$.
Observe that $\tau I_{y_{n_2}} = S_{x_1}$ and that we have the following situation in $\Gamma_A$:

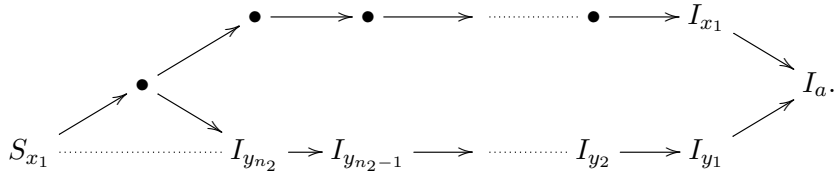
$$\label{caminos-commutatividad2}
 \xymatrix @!0 @R=0.9cm  @C=1.5cm  { &  &  \bullet \ar[r] & \bullet \ar[r]  & \ar@{.}[r] & \ar[r]\bullet & I_{x_1}\ar[rd] & &\\
  &\bullet \ar[rd]\ar[ru]&&& & & & I_{a}. &\\
  S_{x_1} \ar[ru]\ar@{.}[rr]   &  & I_{y_{n_2}}\ar[r] & I_{y_{{n_2}-1}}\ar[r]  & \ar@{.}[r]&\ar[r]I_{y_{2}} &  I_{y_{1}} \ar[ru] & &\\}
$$
\vspace{.05in}

\noindent Since $\ell_{A}(I_{y_{n_2}} \rightsquigarrow I_{a})=n_2$ then $\ell_{A}(S_{x_1} \rightsquigarrow I_{x_1})=n_2+1$.

Now, we show that  $\ell_{A}(P_{x_{1}} \rightsquigarrow S_{x_{1}})=n_1+n_2$. Note that $\tau I_{y_{n_2}}=S_{x_1}$ and that $\tau S_{x_{n_{1}}}=P_{y_{1}}$. Moreover, by \cite[Proposition 1.5]{LM} we have that $\tau S_{x_{i}}=S_{x_{i+1}}$ for $i=1, \dots, n_{1}-1$.

First we prove that the modules $\tau^{-n_1} P_j$ for $j \in \{1, \dots, n_2, b \}$ and the modules $\tau^{-(k-1)} P_k$ for $k \in \{2,  \dots, n_1\}$   are defined.
Suppose that there exists $j \in \{1, \dots, n_2, b \}$ such that $\tau^{-k} P_j$ is injective with $k < n_1$. Consider $j$ to be the less integer %in $\{1, \dots, n_2, b \}$
such that $\tau^{-k} P_j$ is defined for all $j$. We denote such injective by $I$. Since there is a path of irreducible morphisms between indecomposable projective $A$-modules from $P_b$ to $P_{y_{1}}$ of length $n_2$, then there is also a path from $\tau^{-k} P_{b}$ to $\tau^{-k} P_{y_{1}}$  of the same length.

On the other hand, since $k < n_1$ then $\tau^{-k} P_{y_{1}} \simeq  S_{x_{n_{1}}-k}$ and therefore, $I$ is a predecessor of $I_{y_{n_2}}$. Hence, $I \not \simeq I_{y_{l}}$ with $l \in \{1, \dots, n_2, b \}$. Thus $I = I_{x_{h}}$ for $h \in \{1, \dots, n_1 \}$. Since there is a unique path of irreducible morphisms between indecomposable injective $A$-modules from $I_{x_{n_l}}$ to $I_{x_{l}}$   and there is a path between indecomposable $A$-modules between
$I_{x_{h}}$ to $\tau^{-k} P_{y_{1}}$ then we infer that $\tau^{-k} P_{y_{1}}$ is injective. Moreover, $S_{x_{n}-k} \simeq \tau^{-k} P_{y_{1}}$ getting a contradiction  because $S_{x_{n}-k}$ is not injective.

With a similar argument as above we get that the $A$-modules  $\tau^{-(k-1)} P_k$ for $2 \leq k \leq n_1$ are defined. By the existence of a  path of irreducible morphisms from $P_b$ to $P_{x_{1}}$ of length $n_1$ we infer the existence of a path of irreducible morphisms from  $P_{x_{n}}$ to  $\tau ^{-n_{1}} P_{b}$ also of length $n_1$.

On the other hand, by the existence of the  path of irreducible morphisms from $P_b$ to $P_{y_{1}}$ of length $n_2$ we infer the existence of a path of irreducible morphisms from $\tau ^{-n_{1}} P_{b}$ to $\tau ^{-n_{1}} P_{y_{1}} \simeq S_{x_{1}}$ of the same length. Hence, $\ell_{A}(P_{x_{1}} \rightsquigarrow S_{x_{1}})=n_1+n_2$.
\end{proof}

\subsection{Application to monomial tree algebras} \label{A1A2}
The aim of this section is to read the nilpotency index of the radical of the module category of some representation-finite tree monomial algebra from the bounded quiver of $A= kQ_A/ I_A$, in case $I_A$ has not overlapped relations.

We shall describe an algorithm that gives a way to control the category $\mbox{mod}\, A$ when  we delete certain arrows of $Q_A$. This method allows us to analyze $\mbox{mod}\, A$  by simpler quivers whose associated bound quiver algebra may have suitable properties such as  being hereditary algebras, where  the length of certain paths are well-known.
\vspace{.05in}

We start by considering $A$ to be a representation-finite algebra of tree type with only one zero-relation, $\alpha_m \dots \alpha_1$.
We denote by $A_1$ the maximal connected hereditary algebra whose quiver is obtained from $Q_A$ by deleting the arrow $\alpha_1$, and  containing the arrow $\alpha_m$. We denote by $A_2$  the maximal connected hereditary algebra whose quiver is obtained from $Q_A$ by deleting the arrow $\alpha_m$,  and containing  the arrow $\alpha_1$.
\vspace{.05in}

We observe that, with  the above construction the algebra $A$ is the linearly oriented pullback of $A_1 \twoheadrightarrow A_1 \cap A_2$ and $A_2 \twoheadrightarrow A_1 \cap A_2$.

On the other hand, note that $A_1$ and $A_2$ can be obtained from $A_1\cap A_2$ by a finite sequence of one-point extensions or/and  one-point coextensions. Precisely, the first step to get $A_1$ is a one-point extension of $A_1\cap A_2$ by the projective
corresponding to the vertex $e(\alpha_1)$ and the first step to  get $A_2$ is a one-point coextension of $A_1\cap A_2$ by  the injective corresponding to the vertex $s(\alpha_m)$.
Moreover, we can also  obtain $A$ from one-point extensions or/and  one-point coextensions of $A_1$ and $A_2$.

\begin{obs}\label{Intersection-vertices} \emph{ Consider $A$  a representation-finite monomial algebra.
By Theorem \ref{vertex-nilpo}, to determine $r_A$ when $\Gamma_A$ is with length it is enough to consider for each different zero-relation of $I_A$ only one vertex involved in it. In case,  we consider a representation-finite tree algebra with only one zero-relation then it is not hard to see, with similar techniques as in the proof of Lemma \ref{zero-vertex},  that any vertex  of $Q_{A_1 \cap A_2}$, where $A_1 \cap A_2$ is the algebra  constructed as above,  is sufficient to determine $r_A$.}
\end{obs}

In all these section, we refer to the  algebras $A_1$ and $A_2$  as we define above.
\vspace{.05in}

The linearly oriented pullback algebras, are those where the injective and projective $A$-modules can be determined by those ones over $A_1$ and $A_2$.
We recall the following lemma from \cite{CW}.

\begin{lema}\cite[Lemma 2]{CW}\label{proyiny}
Let $A$ be the linearly oriented pullback of $A_1\fepi A_1\cap A_2$ and  $A_2\fepi A_1\cap A_2$. Then
\begin{enumerate}
\item For any $x\in (Q_{A_2})_0$, $P_x$ is the projective module corresponding to the vertex $x$ in $\emph{mod}\, A$ and $\emph{mod}\, A_2.$
\item For any $x\in (Q_{A_1})_0 \backslash (Q_{A_1\cap A_2})_0$, $P_x$ is the projective module corresponding to the vertex $x$ in $\emph{mod}\, A$ and $\emph{mod}\, A_1 .$
\item For any $x \in (Q_{A_1\cap A_2})_0$,  $P_x$ is the projective module corresponding to the vertex $x$ in $\emph{mod}\, A_1$ and $\emph{mod}\, (A_1 \cap A_2).$
\item For any $x\in (Q_{A_1})_0$,  $I_x$ is the injective module corresponding to the vertex $x$ in $\emph{mod}\, A$ and $\emph{mod}\, A_1 .$
\item For any $x\in (Q_{A_2})_0 \backslash (Q_{A_1\cap A_2})_0$, $I_x$ is the injective module corresponding to the vertex $x$ in $\emph{mod}\, A$ and $\emph{mod}\, A_2.$
\item For any $x\in (Q_{A_1\cap A_2})_0$, $I_x$ is the injective module corresponding to the vertex $x$ in $\emph{mod}\, A_2$ and $\emph{mod}\, (A_1 \cap A_2).$
\end{enumerate}
\end{lema}

Let $A$ be the linearly oriented pullback of $A_1\fepi A_1\cap A_2$ and  $A_2\fepi A_1\cap A_2$.
Following \cite{CW}, we denote by $C'$ the direct sum of the projective $A$-modules
$P_i$ with $i\in (Q_{A_1})_0\backslash (Q_{A_1\cap A_2})_0$, and by $DA'$ the direct sum of the injective $A$-modules
$I_i$ with $i\in (Q_{A_2})_0\backslash (Q_{A_1\cap A_2})_0$.

We also consider the following subcategories of $\mbox{ind} \,A$;
$\mathcal{A}=\mbox{Pred}\,DA'$; $\mathcal{C}=\mbox{Succ}\, C'$ and $\mathcal{B}=\mbiA \backslash (\mathcal{A}\cup \mathcal{C})$, and the full subquivers  of $\Gamma_A$ denoted by  $\Sigma'_{A_1}$, $\Sigma'_{A_2}$ and  $\Sigma'_{A_1\cap A_2}$  whose vertices are the objects
of $\mbiA\backslash \mathcal{A}$, $\mbiA\backslash \mathcal{C}$  and $\mathcal{B}$, respectively.
\vspace{.05in}

Next, we recall the following result useful for our purposes.

\begin{lema}\label{sigmas} \cite[Lemma 8]{CW}
Let $A$ be the linearly oriented pullback of $A_1\fepi A_1\cap A_2$ and  $A_2\fepi A_1\cap A_2$. Let  $\Sigma'_{A_1}$, $\Sigma'_{ A_2}$ and $\Sigma'_{A_1\cap A_2}$ be the subquivers of $\Gamma_A$ defined as above. Then the following conditions hold.
\begin{enumerate}
\item $\Sigma'_{A_1\cap A_2}$ is a full and convex subquiver of $\Gamma_{A_1\cap A_2}$, $\Gamma_{A_1}$ and $\Gamma_{A_2}$.
\item $\Sigma'_{A_1}$ is a full and convex subquiver of $\Gamma_{A_1}$.
\item $\Sigma'_{A_2}$ is a full and convex subquiver of $\Gamma_{A_2}$.
\end{enumerate}
\end{lema}

From now on, we consider the category $\mathcal{B} \neq \emptyset$. Moreover, we consider $\mathcal{P}_{\mathcal{B}}$  to be  the full subcategory of $\mbox{mod}\, A_1$ of the predecessors of the modules in $\mathcal{B}$ and $\mathcal{S}_{\mathcal{B}}$ to be
the full subcategory of $\mbox{mod}\ A_2$ of the successors of the modules in $\mathcal{B}$.

We also consider $\Omega_{A_1}$  to be the full subquiver of $\Gamma_{A_1}$  where the vertices are the objects of $\mathcal{P}_{\mathcal{B}}$,   and   $\Omega_{A_2}$ the full subquiver of $\Gamma_{A_2}$ where the vertices are the objects of $\mathcal{S}_{\mathcal{B}}$.

\begin{prop}\label{omegas}
Let $A$ be the linearly oriented pullback of $A_1\fepi A_1\cap A_2$ and  $A_2\fepi A_1\cap A_2$.
Consider $\Omega_{A_1}$ and $\Omega_{A_2}$ as giving above. Then the following statements hold.
\benu
\item[(a)] $\Omega_{A_1}$ is a full and convex subquiver of $\Gamma_{A_1\cap A_2}$.
\item[(b)] $\Omega_{A_2}$ is a full and convex subquiver of $\Gamma_{A_1\cap A_2}$.
\enu
\end{prop}

\begin{proof} We only prove  statement $(a)$, since $(b)$ follows similarly.

First, we prove that the $A_1$-modules of $\mathcal{P}_{\mathcal{B}}$, so are $A_1\cap A_2$-modules.
In fact, let $M \in \mathcal{P}_{\mathcal{B}}$, and
assume that $M\in \mbox{ind}\,(A_1)\backslash \mbox{ind}\,(A_1\cap A_2)$. By \cite[Lemma 5 (ii)]{CW},
$M\in \mathcal{C}$. Moreover, since $M \in \mathcal{P}_{\mathcal{B}}$ there exists an indecomposable successor of $M$, let say  $N\in {\mathcal{B}}$. Since  ${\mathcal{C}}$ is closed under successors,
then $N\in {\mathcal{C}} $ which is a contradiction. Then, all the indecomposable modules of $\mathcal{P}_{\mathcal{B}}$
belong to $\mbox{ind}\,(A_1\cap A_2)$, and therefore $(\Omega_{A_1})_0\subset (\Gamma_{A_1\cap A_2})_0$.

Secondly, we prove that $\Omega_{A_1}$ is a full  subquiver of $\Gamma_{A_1\cap A_2}$. Consider $f:M\fle N$
to be an irreducible morphism in $\mbox{ind}\,A_1$, with $M,N\in \mathcal{P}_{\mathcal{B}}$. Let us
prove that $f$ is irreducible in $\mbox{mod}\,(A_1\cap A_2).$
By \cite[Lemma 6]{CW},  $f$ is neither a split monomorphism nor a split epimorphism in $\mbox{mod}\,(A_1\cap A_2).$
Suppose that $f=gh$ with $h:M \rightarrow L$ and $g:L \rightarrow N$ morphisms in
$\mbox{mod}\,(A_1\cap A_2)$. Then $h$ and $g$ so are morphisms in $\mbox{mod}\,A_1$,
since $\mbox{mod}\,(A_1\cap A_2)$ is a full subcategory of $\mbox{mod}\,A_1$.
Therefore, either $h$ is a split monomorphism in $\mbox{mod}\,A_1$ or $g$
is a split epimorphism in $\mbox{mod}\,A_1,$ and by \cite[Lemma 6]{CW}
so they are in $\mbox{mod}\,(A_1\cap A_2)$. Thus, every arrow of $\Omega_{A_1}$
is an arrow of $\Gamma_{A_1\cap A_2},$ proving that $\Omega_{A_1}$ is a subquiver
of $\Gamma_{A_1\cap A_2}$.

Now, we concentrate to prove that $\Omega_{A_1}$ is a full subquiver of $\Gamma_{A_1\cap A_2}$. Let $f:M \rightarrow N$ be an irreducible morphism in
$\mbox{mod}\,(A_1\cap A_2)$, with $M, N\in \mathcal{P}_{\mathcal{B}}$. We prove that $f$ is irreducible in $\mbox{mod}\,A_1$. Indeed,
$f$ is neither a split monomorphism nor a split epimorphism in $\mbox{mod}\,(A_1\cap A_2)$,
so neither it is in $\mbox{mod}\,A_1$.
Let $L\in \mbox{mod}\,A_1$, $h:M \rightarrow L$ and $g:L \rightarrow N$ be morphisms in
$\mbox{mod}\,A_1$ such that $f=gh$.

We claim that $L$ belongs to
$\mbox{mod}\,(A_1\cap A_2)$. In fact, otherwise, by \cite[Lemma 5 (ii)]{CW}
we have that $L\in \mathcal{C}$. Moreover, since $N\in \mathcal{P}_{\mathcal{B}}$
there is an indecomposable $A_1$-module $X$, such that $X\in \mathcal{B}$ and $X$
is a successor of $N$. Therefore, $X$ is a successor of $L$ and since $\mathcal{C}$ is closed under successors, we infer that $X \in \mathcal{C}$,
which is a contradiction. Thus, $L\in \mbox{mod}\,(A_1\cap A_2)$.

On the other hand, since $f$ is irreducible in $\mbox{mod}\,(A_1\cap A_2)$, then
either $h$ is a split monomorphism in $\mbox{mod}\,(A_1\cap A_2)$ or $g$
is a split epimorphism in $\mbox{mod}\,(A_1\cap A_2),$ and hence, by \cite[Lemma 6]{CW}
so they are in $\mbox{mod}\,A_1$.

Finally, the convexity of $\Omega_{A_1}$ follows from the convexity of $\mathcal{P}_{\mathcal{B}}$
in $\mbox{mod}\,A_1$.
\end{proof}

Now, we are in position to prove one of the main results of this section.

\begin{teo} \label{nilpo}
Let $A$ be the linearly oriented pullback of $A_1\fepi A_1\cap A_2$ and  $A_2\fepi A_1\cap A_2$, with $A_1$ and $A_2$
hereditary algebras, and assume that $\Gamma_A$ is with length.
Let $\mathcal{A}$, $\mathcal{C}$ and $\mathcal{B}$ be the subcategories of $\emph{mod}\, A$ defined above.
If $\mathcal{B}\neq \emptyset$ then $r_A=r_{A_1}+r_{A_2}-r_{A_1\cap A_2}$.
\end{teo}

\begin{proof}
Let $M$ be an indecomposable $A$-module in  $\mathcal{B}$ and
$P_a$ be an indecomposable direct summand of the projective cover of $M$.
Then, there exist non-zero morphisms $f:P_a \rightarrow M$ and $g:M \rightarrow I_a$.
Since $A$ is representation-finite, there exist paths of irreducible
morphisms between indecomposable modules $\phi:P_a\rightsquigarrow M$ and
$\psi: M\rightsquigarrow I_a$  in $\mbox{mod}\, A$.
Since $M \in \mbox{mod}(A_1\cap A_2)$, then $a\in (Q_{A_1\cap A_2})_0$.
By Remark \ref{Intersection-vertices}, the nilpotency index of
$\Re(\mbox{mod}\, A)$ is equal to $\ell(P_a \rightsquigarrow S_a\rightsquigarrow I_a)+1.$
Moreover, since $\Gamma_A$ is with length, then
\begin{equation}\label{rA}
r_A=\ell(\phi)+\ell(\psi)+1.
\end{equation}

Since $M\in \mathcal{B}$ then $M\in (\Sigma'_{A_2})_0$. On the other hand, $P_a\in  (\Sigma'_{A_2})_0$,
otherwise, $P_a\in \mathcal{C}$
and since $M$ is a successor of $P_a$, this implies that $M\in \mathcal{C}$ which is a contradiction.
Hence, since $\Sigma'_{A_2}$ is a convex subquiver of $\Gamma_A$, then the path $\phi$
is in $\sum'_{A_2}$. By Lemma \ref{sigmas} we have that  $\phi:P_a\rightsquigarrow M$
is also a path  in $\Gamma_{A_2}.$
Moreover, since $a\in (Q_{A_2})_0$, by Lemma
\ref{proyiny} we have that $P_a$ is the projective
$A_2$-module corresponding to the vertex $a$.

Let $\rho: M\rightsquigarrow I'_a$ be a path of irreducible morphisms in $\mbox{mod}\, A_2$,
where $I'_a$ is the injective $A_2$-module corresponding to the vertex $a$.
Since $A_2$ is hereditary then
\begin{equation}\label{rA2}
r_{A_2}=\ell(\phi)+\ell( \rho)+1.
\end{equation}

By construction, $\rho$ belongs to $\Omega_{A_2}$ since $M\in \mathcal{B}$, and moreover by Lemma \ref{omegas} $\rho$ is a path of irreducible morphisms
between indecomposable modules in $\mbox{mod}\,(A_1\cap A_2)$. By
Lemma \ref{proyiny}, the module $I'_a$ is the injective $(A_1\cap A_2)$-module
corresponding to the vertex $a$.

With similar arguments as above, the path $\psi: M\rightsquigarrow I_a$ in $\mbox{mod}\,A$
is also a path of irreducible morphisms between indecomposable modules in $\mbox{mod}\,A_1$ and
the module $I_a$ is the injective $A_1$-module corresponding to the vertex $a$.
If we consider $\theta:P'_a\rightsquigarrow M$ a path in $\Gamma_{A_1}$, where $P'_a$
is the projective $A_1$-module corresponding to the vertex $a$, we have that

\begin{equation}\label{rA1}
r_{A_1}=\ell(\theta)+\ell(\psi)+1,
\end{equation}

\noindent because $A_1$ is a hereditary algebra.
Moreover, the path $\theta$ is also a path in $\Gamma_{A_1\cap A_2}$ and the module
$P'_a$ is the projective $(A_1\cap A_2)$-module corresponding to the vertex $a$.
Therefore, since the algebra $A_1\cap A_2$ is hereditary, we also have that

\begin{equation}\label{rA1A2}
r_{A_1\cap A_2}=\ell(\theta)+\ell(\rho)+1.
\end{equation}

From the equalities (\ref{rA}), (\ref{rA2}), (\ref{rA1}) and (\ref{rA1A2}) we obtain that $r_A=r_{A_1}+r_{A_2}-r_{A_1\cap A_2}$.
\end{proof}

\begin{coro}
Let $A$ be a monomial tree algebra with only one zero-relation.%  of length $m$, with $m \geq 3$.
Let $A_1$ and  $A_2$ be defined as in Section \ref{A1A2}.
If $\mathcal{B}\neq \emptyset$, then $r_A=r_{A_1}+r_{A_2}-r_{A_1\cap A_2}$.
\end{coro}

The next example shows that $\mathcal{B}\neq \emptyset$ is a necessary condition in  Theorem \ref{nilpo} to be true.

\begin{ej} \emph{ Consider the algebra $A=kQ_A/I_A$ given  by the following presentation}

$${\xymatrix   @R=0.4cm  @C=0.9cm {
&&1\ar[d]_{\alpha}&&\\
7\ar[r]&4&2\ar[r]_{\beta}\ar[l]&5\ar@{--}@/^/[ul]&6\ar[l]\\
&&3\ar[u]&&}\\}$$

\noindent \emph{where $I_A=<\beta\alpha>$.} \emph{ Computing the Auslander-Reiten quiver of $\mbox{mod}\,A$ one can see that $\mathcal{B}=\emptyset$ and that the length of any path from $P_2$ to $I_2$ is equal to 16. Hence $r_A=17$ but $r_{A_1}+r_{A_2}-r_{A_1\cap A_2}=13$, where $A_1$ is the hereditary algebra given by the quiver }

$${\xymatrix   @R=0.3cm  @C=0.6cm {
&&1\ar[d]&\\
7\ar[r]&4&2\ar[l]&3\ar[l]\\
}}$$

\noindent \emph{ $A_2$ is the hereditary algebra given by the quiver }

$${\xymatrix   @R=0.3cm  @C=0.6cm {
7\ar[r]&4&2\ar[r]\ar[l]&5&6\ar[l]\\
&&3\ar[u]&&}}$$

\noindent \emph{ and $A_1\cap A_2$ is the algebra }

$${\xymatrix   @R=0.3cm  @C=0.6cm {
7\ar[r]&4&2\ar[l]&3\ar[l]
}.}$$
\end{ej}

We dedicate the remaining  part of this paper to find monomial tree algebras where $\mathcal{B}\neq \emptyset$.

We start proving a useful  relationship between the categories $\mathcal{A} \cap \mathcal{C}$ and   $\mathcal{B}$.

\begin{teo} \label{categ}
Let $A$ be the linearly oriented pullback of $A_1\fepi A_1\cap A_2$ and  $A_2\fepi A_1\cap A_2$, with $A_1$ and $A_2$
hereditary algebras, and assume that $\Gamma_A$ is with length.
Let $\mathcal{A}$, $\mathcal{C}$ and $\mathcal{B}$ be the subcategories of $\emph{ind}\, A$ defined above.
If  $\mathcal{A} \cap \mathcal{C} = \emptyset$ then  $\mathcal{B}\neq \emptyset$.
\end{teo}

\begin{proof} Let $P$ be a direct summand of $\mathcal{C}'$ and $I$ a direct summand of ${DA}'$,
such that no summand of $\mbox{rad}\, P$ is in $\mathcal{C}$ and no summand of $I/\mbox{soc}\,I$
is in $\mathcal{A}$. Consider $\mbox{rad}\, P= \oplus_{i=1}^{r} R_i$, where $R_i$ is an indecomposable direct summand of $\mbox{rad}\, P$. We prove that $R_i \in \mathcal{B}$ for some $i=1, \dots, r$.  Since $\mbox{Hom}_A(P, I) \, =0$ then  $\mbox{Hom}_A(\mbox{rad}\, P, I) \, =0$.
%Clearly, for each $i=1, \dots, r$ the module $R_i \notin \mathcal{C}$, otherwise there is a cycle in $\Gamma_A$.

Without loss of generality, consider that $R_1 \in \mathcal{A}$. Then there is a path of irreducible morphisms between indecomposable modules as follows
$R_1 \rightarrow X_1 \rightarrow \dots \rightarrow X_{n}  \rightarrow I$. Note that such a path  is not sectional, since $\mbox{Hom}_A(\mbox{rad}\, P, I) \, =0$. Therefore, for some  $i= 1,  \dots, n$, $X_{i+1} \simeq \tau^{-1} X_{i-1}$. Consider $s \in \{1, \dots, n \}$ to be the least integer such that $R_1 \rightarrow X_1 \rightarrow \dots \rightarrow X_{s}$  is sectional. We claim that none of the $X_i$ with $i=1, \dots, s$ is injective. In fact, otherwise, if for some $i \in \{1, \dots, n \}$ we have that $X_i=I'$, where $I'$ is an indecomposable injective then $\mbox{Hom}_A(R_1, I') \, \neq 0$. Thus $I' \in \mathcal{A} \cap \mathcal{C}$ contradicting the hypothesis.

On the other hand, since $X_1 \not \simeq P$ the we can build the path $P\rightarrow \tau^{-1}R_1 \rightarrow \tau^{-1} X_1 \rightarrow \dots \rightarrow \tau^{-1}X_{s-1}$ and in consequence  $\tau^{-1}X_{s-1} \in \mathcal{A} \cap \mathcal{C}$ contradicting the hypothesis. Hence, $R_1 \in \mathcal{B}$, proving the result.
\end{proof}

Now, we prove some technical results for our further purposes.

\begin{lema}\label{bypass}
Let $A$ be a representation-finite algebra and $\Gamma_A$ a component without cycles (not necessarily with length). Let  $P$ and $I$ be indecomposable  projective and injective $A$-modules, respectively. Let $M$ be an indecomposable $A$-module. The following conditions hold.
\begin{enumerate}
\item If there is an irreducible morphism from a direct summand of $\emph{rad}\, P$ to $M$ with $M \not \simeq P$ or if $M$ is a direct summand of $\emph{rad}\, P$  then $M$ is not a successor of $P$.
\item If there is an irreducible morphism from  $M$ to a direct summand of $I/\emph{soc}\, I$ with $M \not \simeq I$ or if $M$ is a direct summand of $I/\emph{soc}\, I$ then $M$ is not a predecessor of $I$.
\end{enumerate}
\end{lema}
\begin{proof} We only prove (1)  since (2) follows similarly.
Assume there is an irreducible morphism from  $R_1 \rightarrow M$, where $R_1$ is a direct summand $\mbox{rad}\, P$ and $M$ is a successor of $P$.  Then, there is a path of irreducible morphism from $P$ to  $M$, and therefore from $R_1$ to $M$ let say $R_1 \rightarrow P \rightarrow Y_1 \rightarrow \dots \rightarrow Y_n  \rightarrow M$ .  Moreover, that path is a bypass of the arrow $R_1 \rightarrow M$,  since by hypothesis  $M \not \simeq P$ and  $Y_n \not \simeq P$
because $\Gamma_A$ is directed.
By \cite{CHR}, Crawley-Boevey, Happel and Ringel proved that any bypass of an arrow in a component without oriented cycles of $\Gamma_A$ is sectional.

On the other hand, the mentioned authors proved that no arrow in $\Gamma_A$ allows any sectional bypass in case $A$ is representation-finite. Hence, $M$ is not a successor of $P$.

Now, if $M= R_1$  and we assume that  $M$ is a successor of $P$ then we get to the contradiction that there is cycle in $\Gamma_A$. Hence, $M$ is not a successor of $P$.
\end{proof}

\begin{lema}\label{9}
Let $A$ be a representation-finite tree algebra with only one zero-relation $\alpha_m \dots \alpha_1$, where $s(\alpha_1)=a$, $e(\alpha_m)= b$ and  $m \geq 2$. Let  $M$ be an indecomposable  $A_1 \cap A_2$-module.
\begin{enumerate}
\item If there is a sectional path in $\Gamma_A$ from a direct summand of $\emph{rad}\, P_a$ to $M$ then $M \notin \mathcal{C}$.
\item If there is a sectional path in $\Gamma_A$ from $M$ to a direct summand of $I_b/\emph{soc}\, I_b$ then $M \notin \mathcal{A}$.
\end{enumerate}
\end{lema}
\begin{proof} We only prove (1) since (2) follows with similar arguments. Let $\oplus_{i=1}^{r} R_i$ be the decomposition of $\mbox{rad}\, P_a$ in indecomposable direct summands. Without loss of generality, we may assume that the given sectional path goes from $R_1$ to $M$. Consider $n$ to be the length of that path.

Assume that $M$  is a successor of $P_a$. Hence there is a path of irreducible morphisms between indecomposable modules from $P_a$ to $M$.
We analyze the length of the given sectional path.
If $n=0$ or $n=1$ we get a contradiction to Lemma \ref{bypass}.

Consider $n > 1$. Observe that $P_a$ does not belong to the given sectional path, since $\mbox{Hom}_A(P_a, M) \, =0$. %because $M \in A_1 \cap A_2$ and $P_a \in A_1$.
By our assumption that $M$  is a successor of $P_a$, then there are at least two paths of irreducible morphism from $R_1$   to  $M$, the sectional one and another which we illustrate below:

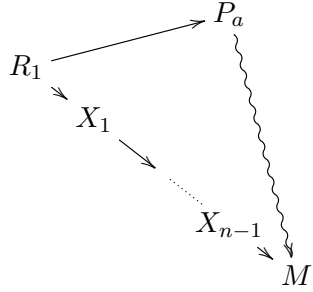
\begin{equation}\label{orbitgraph1}
 \xymatrix @!0 @R=0.7cm  @C=0.9cm  {&  & &  & P_a \ar@{~>}[rddddd] & & & &\\
  &R_1  \ar[rd]\ar[rrru]  & & & && &  \\
   &  & X_{1}  \ar[rd] &  & &   &\\
   &  &  & \ar@{.}[rd]   & & &  &\\
   &  &   &   &X_{n-1}\ar[rd] & & &\\
   &  &  &  & & M & &}
\end{equation}

\noindent both of the same length. We analyze the orbit graph of the subquiver in  (\ref{orbitgraph1}) (see,  \ref{orbit}). We have the following subgraph

$$
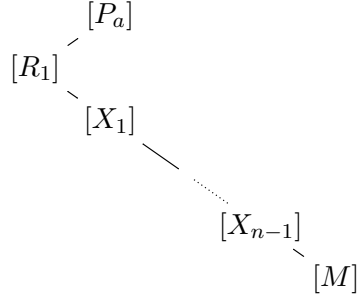
\label{orbitgraph2}
 \xymatrix @!0 @R=0.7cm  @C=1.0cm  {&  & [P_a]  &  & & & &\\
  &[R_1] \ar@^{-}[rd] \ar@^{-}[ru] & &  & & &&  \\
   &  & [X_{1}]  \ar@^{-}[rd] &  & &   &\\
   &  &  & \ar@{.}[rd]   & & &  &\\
   &  &   &   &[X_{n-1}] \ar@^{-}[rd] & & &\\
   &  &  &  & &[M] & & }$$

Consider the path from $R_1$ to $M$ going through $P_a$ as follows: $R_1 \rightarrow P_a \rightarrow Y_1 \rightarrow \dots \rightarrow Y_{n-2}  \rightarrow  M$. Note that we have exactly $n-3$ vertices that we have to consider in the orbit graph. The walk between the vertices $[X_{1}], \dots,   [X_{n-1}]$  can not  coincide with the one of the vertices $[Y_{1}], \dots,   [Y_{n-2}]$ since the number of vertices is different.  Therefore, we get that the orbit graph is not of tree-type, a contradiction to \cite[Lemma 4.3 and Lemma 4.8, Chapter IX]{ASS}, since $A$ is a tree algebra. Hence, $M \notin \mathcal{C}$.
\end{proof}

Our next result is fundamental to get Theorem \ref{m-relations}.

\begin{prop}\label{4}
Let $A$ be a representation-finite tree algebra with only one zero-relation
$\alpha_m \dots \alpha_1$, where $s(\alpha_1)=a$, $e(\alpha_m)= b$ and  $m \geq 2$.
If there is a sectional path from a direct summand of $\emph{rad}\, P_a$ to a direct summand
of $I_b/\emph{soc}\, I_b$ then $\mathcal{A}\cap \mathcal{C}= \emptyset$.
\end{prop}

\begin{proof}
Let $R_1$ and $J_1$ be direct summands of $\mbox{rad}\,P_a$ and $I_b/\mbox{soc}\, I_b$, respectively.
Assume that there is a sectional path $R_1\fle X_1\fle \cdots \fle X_{n-1}\fle J_1$ of length $n$.
Since $\mbox{Hom}_A(P_a, J_1) \, =0$ and $\mbox{Hom}_A(R_1, I_b) \, =0$ we infer that the modules $P_a$ and $I_b$
do not belong to the sectional path. Analyzing the orbit graph of $\Gamma_A$ we have the following subgraph:

$$
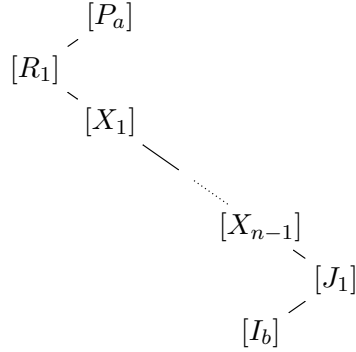
\label{orbitgraph3}
 \xymatrix @!0 @R=0.7cm  @C=1.0cm  {&  & [P_a]  &  & & & &\\
  &[R_1] \ar@^{-}[rd] \ar@^{-}[ru] & &  & & &&  \\
   &  & [X_{1}]  \ar@^{-}[rd] &  & &   &\\
   &  &  & \ar@{.}[rd]   & & &  &\\
   &  &   &   &[X_{n-1}] \ar@^{-}[rd] & & &\\
   &  &  &  & &[J_1] & & \\
    &  &   &   &[I_b] \ar@^{-}[ru] & & &}$$

Suppose  that $\mathcal{A} \cap \mathcal{C} \neq \emptyset$. Then there exists an indecomposable module $M$ and
paths $P_a\rightsquigarrow M$ and  $M\rightsquigarrow I_b$. Since $\Gamma_A$ is a component with length, then $\ell(P_a\rightsquigarrow M\rightsquigarrow I_b)=n-2$. With a similar argument as in the proof of Lemma \ref{9}, we conclude that the orbit graph of $\Gamma_A$ is not of tree type, a contradiction to \cite[Lemma 4.3 and Lemma 4.8, Chapter IX]{ASS}, since $A$
is a tree algebra. Therefore, $\mathcal{A}\cap \mathcal{C}= \emptyset$.
\end{proof}

Now, we prove that the possible algebras $A_1\cap A_2$ are hereditary algebras of type $A_n$, $D_n$ or $E_6$. We shall consider them to find different monomial tree algebras with $\mathcal{B} \neq  \emptyset$.

\begin{lema}\label{2}
Let $A$ be a representation-finite tree algebra with only one zero-relation, $a\stackrel{\alpha_1}\rightarrow a_1 \stackrel{\alpha_2}\rightarrow a_2 \rightarrow \dots \stackrel{\alpha_{m-1}}\rightarrow a_{m-1} \stackrel{\alpha_m}\rightarrow b$, with $m \geq 2$. Let $A_1$ and $A_2$ be algebras build as we explained above. The following conditions hold.
\begin{enumerate}
\item If $A_1\cap A_2$ is not of type $A_n$ then the vertices $a_1$ and $a_m$ in $Q_{A_1\cap A_2}$  are the starting and the ending points of only one arrow, respectively.
\item The algebra $A_1\cap A_2$ is neither of type $E_7$ nor of type $E_8$.
\item The algebra $A_1\cap A_2$ is not of type $D_n$ with $n\geq 8$ and $m \geq 3$.
\end{enumerate}
\end{lema}
\begin{proof}
(1) Assume that the vertex $a_1$ in $Q_{A_1\cap A_2}$ has more than one arrow. Then there is a vertex $c$ and an edge with end points $c$ and $a_1$. Since $A_1\cap A_2$ is not of type $A_n$, then there is a vertex $x$ having three edges. Hence, in the extension $(A_1\cap A_2)[P_{a_1}]$ we have a subquiver of type $\widetilde{D}_{n}$ as follows
 
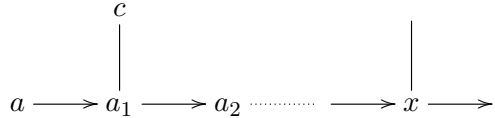
\[{\xymatrix{&c\ar@^{-}[d]&&&&\\
      a\ar[r]&a_1\ar[r]&a_2\ar@^{.}[r]&\ar[r]&x\ar[r]\ar@^{-}[u]&}}\]

\noindent contradicting the fact that $A$ is representation-finite. Similarly, we can prove that $a_m$ has only one arrow in $Q_{A_1\cap A_2}$.

(2) Assume that the algebra $A_1\cap A_2$ is of type $E_8$. Then we have the following situation

 
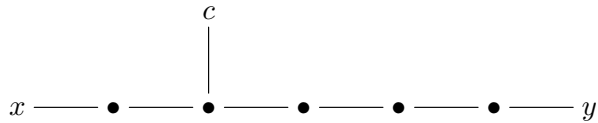
\[{\xymatrix{&&c\ar@^{-}[d]&&&&\\
  x\ar@^{-}[r]&\bullet\ar@^{-}[r]&\bullet\ar@^{-}[r]&\bullet\ar@^{-}[r]&\bullet\ar@^{-}[r]&\bullet\ar@^{-}[r]&y}.}\]

By Statement (1), to get the algebras $A_1$ and $A_2$ we should extend or coextend in the vertices $x$, $y$ or $c$. If we add an edge in the vertex $x$, then the quiver will contain a subquiver of type $\widetilde{E}_{7}$. Now, if we add an edge in $y$, then the quiver  contains a subquiver of type $\widetilde{E}_{8}$. Finally, if we add an edge in the vertex $c$, then the obtained quiver has a subquiver of type $\widetilde{E}_{6}$. Hence, in none of the cases, we can do a extension or a coextension to obtain a representation-finite algebra.

Now, assume that $A_1\cap A_2$ is of type $E_7$. Then, we have a diagram as follows

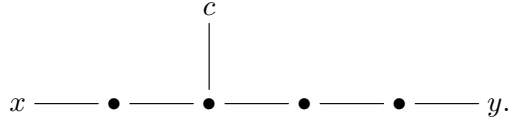
\[{\xymatrix{&&c\ar@^{-}[d]&&&\\
  x\ar@^{-}[r]&\bullet\ar@^{-}[r]&\bullet\ar@^{-}[r]&\bullet\ar@^{-}[r]&\bullet\ar@^{-}[r]&y.}}\]

Again, by  Statement (1) to get the algebras $A_1$ and $A_2$ we should extend or coextend in the vertices $x$, $y$ or $c$.  If we add an edge in the vertex $x$, then such a quiver will contain a subquiver of type $\widetilde{E}_7$, which is a contradiction to the fact of being representation-finite. Now, if we add an edge in the vertex $c$, then the quiver will contain a subquiver of type $\widetilde{E}_6$, getting the same contradiction as above. Finally, the last option is to add an edge in $y$. If $m \geq 3$, then it is not possible to get $A_1$ and $A_2$ both by a one-point extension of a vertex. In case $m = 2$ and if we extend or coextend in $y$ we obtain a graph of the list stated in \cite[Corollary 1]{BR}, which is representation-infinite.
Hence, we conclude that $A_1\cap A_2$ can not be of type $E_7$.

(3) Let us  prove that $A_1\cap A_2$ is not of type $D_n$ with $n\geq 8$ and $m \geq 3$. In fact, assume that $A_1\cap A_2$ is of type $D_8$. Then we have a diagram as follows

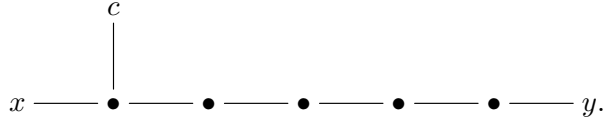
\[{\xymatrix{&c\ar@^{-}[d]&&&&&\\
  x\ar@^{-}[r]&\bullet\ar@^{-}[r]&\bullet\ar@^{-}[r]&\bullet\ar@^{-}[r]&\bullet\ar@^{-}[r]&\bullet\ar@^{-}[r]&y.}}\]

By Statement (1), to get the algebras $A_1$ and $A_2$ we have to extend or coextend in the vertices $x$, $y$ or $c$. Necessarily, we have to  extend in $x$ or in $c$. Either if we add an edge in $x$ or in $c$ we obtain a subquiver of type $\widetilde{E}_{8}$, getting a contradiction to the assumption that $A$ is representation-finite.
\end{proof}

\begin{obs} \emph {Note that if the first step to get $A_1$ from $A_1\cap A_2$ is to extend $A_1\cap A_2$ by a module $X$ then the first step to get $A$ from $A_2$ is to extend $A_2$ by the same module $X$.}

\emph {On the other hand, if the first step to obtain $A$ from $A_1$ is a coextension by a module $Y$ then the first step to get $A$ from $A_2$ is a coextension by a module $Y$.}
\end{obs}

Now, we concentrate to prove some lemmas necessary to find tree algebras where $\mathcal{B} \neq \emptyset$.

\begin{lema}\label{1} Consider  $A$, $A_1$ and $A_1 \cap A_2$  defined as above.
\begin{enumerate}
\item[(a)] Assume that the first step to obtain $A_1$ from $A_1\cap A_2$ is a extension by an indecomposable  module $X$ and that there is an indecomposable module $M$ and a path of irreducible morphisms between indecomposable $A_1 \cap A_2$-modules  as follows
\begin{equation}\label{camino} Z_1=X \rightarrow Z_{2}\rightarrow \dots \rightarrow Z_{k}=M
\end{equation}
\noindent where the path from $X$ to $M$ is sectional. Then the path in $(\ref{camino})$ is also a path of irreducible morphisms between indecomposable modules in $\emph{mod}\,A_1$, and  moreover it is a  sectional path in  $\emph{mod}\,A_1$.
\item[(b)] Assume that the first step to obtain $A$ from $A_1$ is a co-extension by a module $Y$ and that there is an indecomposable module $M$ and a path of irreducible morphisms between indecomposable $A_1$-modules as follows
\begin{equation}\label{camino2} M \rightarrow U_1 \rightarrow \dots \rightarrow U_s=Y
\end{equation}
\noindent where the path from $M$ to $Y$ is sectional. Then the path in  $(\ref{camino2})$ is also a path of irreducible morphisms between indecomposable $A$-modules, and  moreover it is a sectional path in  $\emph{mod}\,A$.
\end{enumerate}
\end{lema}

\begin{proof}
We only prove (a) since (b) follows similarly. By construction  $A_1 \cap A_2$ is a Dynkin algebra.
Since the path from $X$ to $M$ is sectional then  $\mbox{Hom}_{A_1 \cap A_2}(X, \tau_{_{A_1 \cap A_2}} Z_i) = 0$ for $i=1, \dots, k$.
By Proposition \ref{lema S-S}, the path in $(\ref{camino})$  is a path of irreducible morphisms between indecomposable modules in  $(A_1 \cap A_2)[X]$. We denote $(A_1 \cap A_2)[X]$  by $B_1$ and we have two cases to analyze; first if we extend by a new projective $X_{1} = P_X$ and secondly if we coextend by a new injective $Y_1 = I_X$.

In case we extend by the new projective $X_{1} = P_X$ since $\tau_{B_{1}} Z_i \in \mbox{ind}\, (A_1 \cap A_2)$ for $i = 1, \dots, n$, then
$\mbox{Hom}_{B_1}(P_X, \tau_{_{B_1}} Z_i) = 0$ and therefore the path in $(\ref{camino})$ is a path of irreducible morphisms between indecomposable modules in $B_1[X_1]$.

If we coextend by a new injective $Y_{1} = I_X$ since $Z_i \in \mbox{ind}\, (A_1 \cap A_2)$ for $i = 1, \dots, n$, then
$\mbox{Hom}_{B_1}( Z_i, I_X) = 0$. By Proposition \ref{dual lema S-S} the almost split sequence
$$0 \rightarrow \tau_{B_1} Z_i \rightarrow E \rightarrow Z_i \rightarrow 0$$
\noindent in $\mbox{mod}\, B_1$ remains almost split in $\mbox{mod}\, [Y_1]B_1$.
Therefore $(\ref{camino})$ is a path of irreducible morphisms between indecomposable modules in $[Y_1]B_1$. We denote $[Y_1]B_1$ by $B_2$.

Assume that we iterate this process $n$ times (extending and/or coextending by an indecomposable new projective or by an indecomposable new  injective $A$-module, respectively) and that the path in $(\ref{camino})$  is a path of irreducible morphisms between indecomposable modules in $B_n$. By construction we have that $\tau_{B_{n}} Z_i = \tau_{A_1 \cap A_2} Z_i$ for $i = 1, \dots, k$.
Now, since in the $n+1$-step we extend by a new projective $X_n=P_{X_n-1}$ or we coextend by a new injective $Y_n=I_{X_{n-1}}$, and since $Z_i$ and $\tau_{B_n}Z_i$ are in $\mbox{mod}\, (A_1 \cap A_2)$ then $\mbox{Hom}_{B_n}(X_n,\tau_{B_n}Z_i)=\mbox{Hom}_{B_n}(Z_i,Y_n)=0$. In any case we have that  $(\ref{camino})$ is a path of irreducible morphisms in  $B_{n+1}$, which is the extension or the coextension of $B_n$.
Hence, the path in $(\ref{camino})$  is a path of irreducible morphisms between indecomposable modules in $A_1$.
\end{proof}

\begin{lema} \label{10} Consider  $A$, $A_2$ and $A_1 \cap A_2$  defined as above.
\begin{enumerate}
\item[(a)] Assume that the first step to obtain $A_2$ from $A_1\cap A_2$ is a co-extension by a module $Y$ and that there is an indecomposable module $M$ and a path of irreducible morphisms between indecomposable $A_1\cap A_2$-modules as follows
\begin{equation}\label{camino3} M \rightarrow U_1 \rightarrow \dots \rightarrow U_s=Y
\end{equation}
\noindent where the path from $M$ to $Y$ is sectional. Then the path in $(\ref{camino3})$ is a path of irreducible morphisms between indecomposable $A_2$-modules, and  moreover it is also a sectional path in  $\emph{mod}\,A_2$.
\item[(b)] Assume that the first step to obtain $A$ from $ A_2$ is a extension by an indecomposable  module $X$ and that there is an indecomposable module $M$ and a path of irreducible morphisms between indecomposable $A_2$-modules as follows
\begin{equation}\label{camino4} Z_1=X \rightarrow Z_{2}\rightarrow \dots \rightarrow Z_{k}=M
\end{equation}
\noindent where the path from $X$ to $M$ is sectional. Then the path in $(\ref{camino4})$ is a path of irreducible morphisms between indecomposable $A$-modules, and  moreover it is a sectional path in  $\emph{mod}\,A$.
\end{enumerate}
\end{lema}

Now,  we are in position to show some tree algebras where $\mathcal{B} \neq \emptyset$.
We start with the ones where $A_1 \cap A_2$ is a hereditary algebra of type $A_n$, for $n \geq 1$.

\begin{prop} \label{ejemplos} Let $A \simeq kQ_A/I_A$,  where $I_A= <{\alpha_{m}} \dots {\alpha_{1}}>$ with $m \geq 2$ and  $Q_A$ is given by the quiver
\begin{enumerate}
\item
$$\label{bypass0}
 \xymatrix @!0 @R=0.6cm  @C=1.7cm  {
Q_{A'}\ar@^{-}[r] & 1 \ar[r]^{\alpha_1} & 2 \ar[r]^{\alpha_2} & \ar@{.}[r]  & \ar[r]^{\alpha_{m-1}} & m \ar[r]^{\alpha_{m}} & m+1 \ar@^{-}[r]&  Q_{A ''} &
}$$

\noindent where $Q_{A '}$ and $Q_{A''}$ are quivers of any representation-finite hereditary algebras.

\item
$$\label{bypass1}
 \xymatrix @!0 @R=0.8cm  @C=1.7cm  {
Q_{A'}\ar@^{-}[r] & 1 \ar[r]^{\alpha_1} & 2 \ar[r]^{\alpha_2} \ar@^{-}[d] & \ar@{.}[r]  & \ar[r]^{\alpha_{m-1}} & m \ar[r]^{\alpha_{m}} & m+1 \ar@^{-}[r]&  Q_{A ''} &  \\
 &  & Q_{A'''} &  &   & & &}$$

\noindent where, $Q_{A '}$ and $Q_{A''}$ are quivers of any representation-finite hereditary algebras and  $Q_{A '''}$ is the quiver of a hereditary algebra of type $A_n$.

\item
$$\label{bypass2}
 \xymatrix @!0 @R=0.8cm  @C=1.7cm  {
Q_{A'}\ar@^{-}[r] & 1 \ar[r]^{\alpha_1} & 2 \ar[r]^{\alpha_2}  & \ar@{.}[r]  & \ar[r]^{\alpha_{m-1}} & m \ar[r]^{\alpha_{m}} \ar@^{-}[d] & m+1  \ar@^{-}[r]&  Q_{A ''} &  \\
  & & &  &   & Q_{A'''} & &}$$

\noindent where, $Q_{A '}$ and $Q_{A''}$ are quivers of any representation-finite hereditary algebras and  $Q_{A '''}$ is the quiver of  a hereditary algebra of type $A_n$.

\item
$$\label{bypass3}
 \xymatrix @!0 @R=0.8cm  @C=1.7cm  {
 1 \ar[r]^{\alpha_1} & 2 \ar[r]^{\alpha_2} \ar@^{-}[d] & \ar@{.}[r]  & \ar[r]^{\alpha_{m-1}} &m \ar[r]^{\alpha_{m}}\ar@^{-}[d] & m+1 &    \\
   & m+2 &  &   & m+3 & }$$
\end{enumerate}

 Then  $r_{A}= r_{A_1} + r_{A_2} - r_{A_1 \cap A_2}.$
\end{prop}

\begin{proof}  By Theorem \ref{nilpo} it is enough to prove that $\beta \neq \emptyset$.
Statement (1) is an immediate consequence of Lemma \ref{9}, since $\mbox{rad}\, P_1 \simeq I_{m+1}/\mbox{soc}\, I_{m+1}$.

(2) We prove that there is a sectional path from a direct summand of $\mbox{rad} \,  P_{1}$ to a direct summand of $I_{m+1}/\mbox{soc}\, I_{m+1}$. Then by Lemma \ref{9}, we infer that $\beta \neq \emptyset$.

Since $Q_{A '''}$ is a hereditary algebra of type $A_n$, then $A_1 \cap A_2$ is also a hereditary algebra of type $A_n$. It is clear that a direct summand of $\mbox{rad} \,  P_{1}$ is isomorphic to $P_2$ in $\mbox{mod}\, (A_1 \cap A_2)$ and that a  direct summand of $I_{m+1}/\mbox{soc}\, I_{m+1}$ is isomorphic to $I_m$ in $\mbox{mod}\, (A_1 \cap A_2)$.

Since $A_1 \cap A_2$ is a string algebra then from \cite[Proposition 2.5 (a)]{CG} there is a sectional path from $P_2$ to $I_{m}$ in $\mbox{mod}\, (A_1 \cap A_2)$.  Applying Lemma \ref{1} and Lemma \ref{10} we get that there is a sectional path from $P_{2}$ to $I_{m}$ in $\Gamma_{\mbox{mod}\, A}$, proving the result.

(3) The proof of this case is similar to Statement (2).

(4) We only analyze the case where the edges have the orientation  $2\rightarrow m+2$ and $n \rightarrow m+3$. The other cases follow similarly.

By definition, it is not hard to see that $\tau I_{m+1} \simeq P_1$. Hence there is an almost split sequence $0 \rightarrow P_1 \rightarrow M \rightarrow I_{m+1} \rightarrow 0$ where $M \in \mbox{mod}\, A$. Moreover, since $\mbox{Hom}_A(P_{1}, I_{m+1})=0$ then $M$ is indecomposable. In consequence, $\mbox{rad}\, P_{1}$ and $I_{m+1}/\mbox{soc}\, I_{m+1}$ are indecomposable too,  and $M \simeq \tau^{-1} \mbox{rad}\, P_{1}\simeq \tau I_{m+1}/\mbox{soc}\, I_{m+1}$.

Note that the almost split sequences starting in $\mbox{rad}\, P_{1}$ and in $M$ have exactly two indecomposable middle terms. In fact, assume that the almost split sequence starting in $\mbox{rad}\, P_{1}$ is as follows:
 $0 \rightarrow \mbox{rad}\, P_{1} \rightarrow N_1 \rightarrow M \rightarrow 0$. Clearly,  $P_1 \not \simeq N_1$. Now, if $N_1$ decomposes in more than two indecomposable summand then we infer that $\mbox{Hom}_A(\mbox{rad}\, P_{1}, I_{m+1})\neq 0$ a contradiction to the fact that $\mbox{Hom}_A(P_{1}, I_{m+1})$ vanishes. Thus $N_1 = P_1 \oplus X_1$ with $X_1$ indecomposable. With similar arguments we have that  $N_2 = I_{m+1} \oplus X_2$,  with $X_2$ indecomposable.

We illustrate the situation as follows

$$%\label{bypass}
 
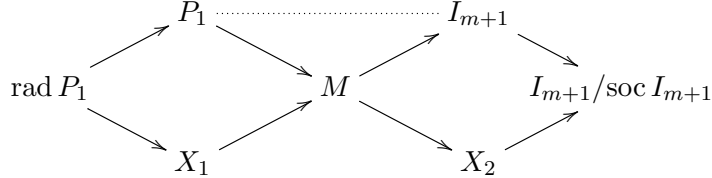
\xymatrix @!0 @R=1.0cm  @C=1.9cm  {
 & P_1 \ar[rd]  \ar@{.}[rr] & & I_{m+1}\ar[rd] &       \\
   \mbox{rad}\, P_{1} \ar[ru]  \ar[rd]  &  & M \ar[ru]  \ar[rd]  &   & I_{m+1}/\mbox{soc}\, I_{m+1}   \\
     & X_1 \ar[ru] &  &    X_2 \ar[ru] &  }$$

Note that $X_1$ is not injective, since  the irreducible morphism from $X_1 \rightarrow M$ is a monomorphism.  Similarly $X_2$ is not projective. Hence, there is an almost split sequence starting in $X_1$. By \cite[Remark 3.5]{Cha}, we know that the almost split sequence starting in $X_1$ has more than one indecomposable middle term. Assume that the middle term of the almost split sequence starting in $X_1$   is $M \oplus M'$,  with $M' \neq 0$.

Consider $N'$ an indecomposable  direct summand of $M'$. Following Lemma \ref{9} we infer that $N' \in  \mathcal{B}$.
Therefore $\mathcal{B} \neq \emptyset$.
\end{proof}

The next result holds when $A_1\cap A_2$ is a hereditary algebra of type $D_n$.

\begin{prop}\label{interDn}
 Let $A \simeq kQ_A/I_A$,  where $Q_A$ is given by the quiver
 \begin{enumerate}
\item
 $$%\label{Dn1}
 \xymatrix @!0 @R=0.8cm  @C=1.7cm  {
Q_{A'}\ar@^{-}[r] & 1 \ar[r]^{\alpha_1} & 2 \ar[r]^{\alpha_2}& 3 \ar[r]^{\alpha_3} \ar@^{-}[d] & \ar@{.}[r]  & \ar[r]^{\alpha_{m}} & m+1 \ar@^{-}[r]&  Q_{A ''} &  \\
 & & & m+2 &  &    & &}$$

 \item
$$\label{Dn2}
 \xymatrix @!0 @R=0.8cm  @C=1.7cm  {
Q_{A'}\ar@^{-}[r] & 1 \ar[r]^{\alpha_1} &  \ar@{.}[r]  & \ar[r]^{\alpha_{m-1}} & m-1 \ar[r]^{\alpha_{m-1}} \ar@^{-}[d] & m \ar[r]^{\alpha_{m}}& m+1  \ar@^{-}[r]&  Q_{A ''} &  \\
  & &   &   &  m+2 & &&}$$

  \end{enumerate}
\noindent where $Q_{A'}$ and $Q_{A''}$ are quivers of any representation-finite hereditary algebras and  $I_A=  <{\alpha_{m}} \dots {\alpha_{1}}>$, with $m \geq 4$. Then  $r_{A}= r_{A_1} + r_{A_2} - r_{A_1 \cap A_2}.$
\end{prop}

\begin{proof} We prove that there is a sectional path in $\mbox{mod}\,A$ from a direct summand of  $\mbox{rad}\,P_1$ to a direct summand of $I_{m+1}/ \mbox{soc}\,I_{m+1}$. By Lemma \ref{9},  we have that $\mathcal{B}\neq \emptyset$, and therefore we get the result.

We only prove Statement (1), and we assume that the edge has the orientation $m+2\rightarrow 3$. The others cases, follow similarly.
Note that $Q_{{A_1\cap A_2}}$ is as follows:

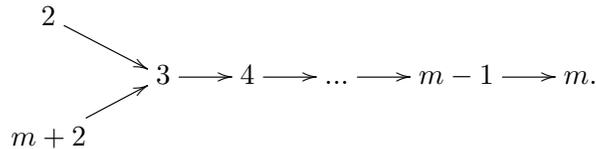
\begin{displaymath}
    \xymatrix  @R=0.3cm  @C=0.7cm {
       2\ar[dr]& &  & & &   \\
       & 3\ar[r] & 4\ar[r]& ...\ar[r] &m-1\ar[r] &m.  \\
       m+2\ar[ur] & & & & & }
\end{displaymath}

We claim that $P_{m}\simeq \tau_{A_1\cap A_2}^{(m-2)}I_m$. In fact, it is not hard to see that $\tau(I_{m})\simeq  S_{3}$.
By \cite{LM}, we know that $\tau(S_{i})\simeq S_{i+1}$ for $i=3, \dots , m$. Hence, $\tau^{(m-2)}I_m\simeq S_m\simeq P_m$.
Moreover, inductively we have that $I_{k}\simeq \tau^{-(m-2)}P_k$, for  $k=3, \dots, m$.

Since $A_1\cap A_2$ is hereditary then there is a sectional path $P_m {\rightarrow }P_{m-1}{\rightarrow } \cdots {\rightarrow } P_{3}{\rightarrow } P_{2}$ between projective $A_1\cap A_2$-modules. Then we also have a sectional path as follows:

\begin{equation} \label{camDn}
P_2 {\rightarrow }\tau^{-1}P_{3}{\rightarrow } \cdots {\rightarrow } \tau^{-(m-3)}P_{m-1}{\rightarrow } \tau^{-(m-2)}P_{m}\simeq I_m.
\end{equation}

Note that $A$ can be obtained by a sequence of one-point extensions and one-point coextensions of $A_1\cap A_2$. Without loss of generality, the first extension can be done in $P_2$ and the first coextension in $I_m$. By Lemma \ref{10}, we know that (\ref{camDn}) is a sectional path in $\mbox{mod}\,A$. The result follows since $P_2$ and $I_m$ are direct summands of $\mbox{rad}\,P_1$ and $I_{m+1}\backslash \mbox{soc}\,I_{m+1}$, respectively.
\end{proof}

\begin{obs} \emph{ In Proposition \ref{ejemplos} (1), (2), (3) and Proposition \ref{interDn}
we prove that there is a sectional path from a direct summand of $\mbox{rad}\,P_1$ to a direct summand of
$I_{m+1}/\mbox{soc}\,I_{m+1}$. Therefore, by Proposition \ref{4} we have that $\mathcal{A}\cap \mathcal{C}=\emptyset$ and
by \cite[Proposition 2]{CW} we know that these algebras are quasitilted.}
\end{obs}

The next result holds when $A_1\cap A_2$ is a hereditary algebra of type $E_6$.

\begin{prop}\label{3}
Let $A$ be a representation-finite tree algebra with only one zero-relation of length $m$,  with $m \geq 3$. Let $A_1$ and $A_2$ be algebras built as we explained above. Suppose that $A_1\cap A_2$ is a hereditary algebra of type $E_6$. Assume that the first step to obtain $A_1$ from $A_1\cap A_2$ is a extension by a module $X$ and  that the first step to obtain $A_2$ from $A_1\cap A_2$ is a coextension by a module $Y$. Then there is an indecomposable module $M$  and sectional paths  from $X$ to $M$ and from $M$ to $Y$ in $\Gamma_{A_1\cap A_2}$. Moreover, there exist sectional paths from $X$ to $M$ and from $M$ to $Y$ in $\Gamma_{A_1}$ and in $\Gamma_{A_2}$.
\end{prop}
\begin{proof} By hypothesis $A_1\cap A_2$ is a hereditary algebra of type $E_6$. Then we have the following subquiver

\[{\xymatrix{&&b\ar@^{-}[d]&&\\
  a\ar@^{-}[r]& a_1\ar@^{-}[r]&a_2 \ar@^{-}[r]&a_3\ar@^{-}[r]&c.}}\]

By Lemma \ref{2} (1), we can extend or coextend in $a$, $b$ or $c$. If we extend or coextend in $b$ then we get a subgraph of type $\widetilde{E}_6$ which is a contradiction, since $A$ is representation-finite. Hence, the unique option is to extend in $a$ and to coextend in $c$ or viceversa.

Without loss of generality, assume that we extend  $A_1\cap A_2$ in $a$ and coextend in $c$. By the possible orientation of the arrows in $Q_{A_1\cap A_2}$ we have two situations to consider. First, if in $Q_{A_1\cap A_2}$ there  is an arrow from $a_2$ to $b$.
Then for $X=P_a$ and $Y=I_b$ there is an indecomposable module $M$ whose composition factors are given by the simples
$${\small{\txt{$a$\;\;\;\,\,\\$a_1$\,\,\,\, \,\,\\$a_2$\,$a_2$\\$a_3$\,b\\c\,\,\,\,\,\,}}},$$
\noindent and sectional paths from $P_a$ to $M$ and from $M$ to $I_c$.

%$\mbox{Hom}_{_{A_1\cap A_2}}(P_a, M)\neq 0$ and $\mbox{Hom}_{_{A_1\cap A_2}}( M, I_c)\neq 0$.

Secondly, if in $Q_{A_1\cap A_2}$ there is an arrow from $b$ to $a_2$ then again for $X=P_a$ and $Y=I_c$ there is an indecomposable module $M$ whose composition factors are $${\small{\txt{$a$\;\;\;\,\,\\$a_1$\,b\\$a_2$\,$a_2$\\$a_3\;\;\;\,\,$\\c\,\,\,\,\,\,}}}$$
 and sectional paths from $P_a$ to $M$ and from $M$ to $I_c$.

%$\mbox{Hom}_{_{A_1\cap A_2}}(P_a, M)\neq 0$ and $\mbox{Hom}_{_{A_1\cap A_2}}( M, I_c)\neq 0$.

It is not hard to see that in both cases $\Gamma_{A_1\cap A_2}$ has a subquiver as follows:

$$
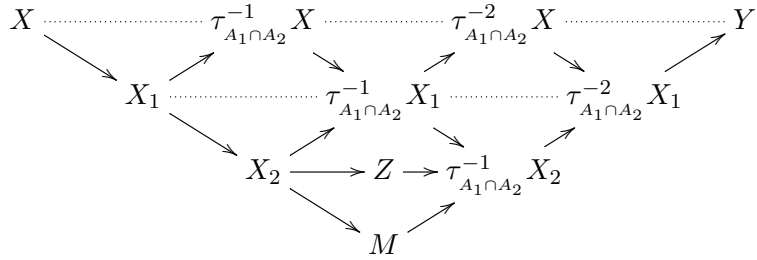
\xymatrix @!0 @R=1.0cm  @C=1.6cm  { X \ar[rd]\ar@{.}[rr]&  & \tau_{_{A_1\cap A_2}}^{-1}X \ar[rd]\ar@{.}[rr]&   & \tau_{_{A_1\cap A_2}}^{-2}X \ar[rd]\ar@{.}[rr] &    & Y  \\
   & X_1 \ar[ru]\ar[rd]\ar@{.}[rr]& & \tau_{_{A_1\cap A_2}}^{-1}X_1 \ar[rd] \ar[ru]\ar@{.}[rr] & & \tau_{_{A_1\cap A_2}}^{-2}X_1 \ar[ru]&  \\
 && X_2 \ar[ru] \ar[r]  \ar[rd] & Z\ar[r] & \tau_{_{A_1\cap A_2}}^{-1}X_2 \ar[ru] &&\\
       & & & M  \ar[ru]& &&  }$$

Note that we can do two consecutive extensions to get $A_1$, such as,  $A_1= [({A_1\cap A_2})[X]][W]$. If we do another extension we find a subquiver of type $\widetilde{E}_8$ getting a contradiction to the fact that $A$ is representation-finite.
By the knitting technique we can obtain in $\Gamma_{({A_1\cap A_2})[X]}$ the following subquiver

$$
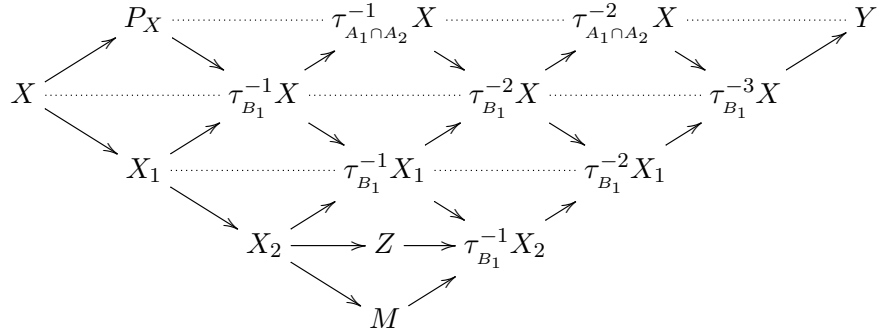
\xymatrix @!0 @R=1.0cm  @C=1.6cm  { & P_X  \ar[rd]\ar@{.}[rr] & & \tau_{_{A_1\cap A_2}}^{-1}X \ar[rd]\ar@{.}[rr]&   & \tau_{_{A_1\cap A_2}}^{-2}X \ar[rd]\ar@{.}[rr] &    & Y  \\
X \ar[ru] \ar[rd]\ar@{.}[rr]&  & \tau_{_{B_1}}^{-1}X \ar[rd]\ar@{.}[rr] \ar[ru]&   & \tau_{_{B_1}}^{-2}X \ar[rd]\ar@{.}[rr] \ar[ru] &    & \tau_{_{B_1}}^{-3}X  \ar[ru] &\\
   & X_1 \ar[ru]\ar[rd]\ar@{.}[rr]& & \tau_{_{B_1}}^{-1}X_1 \ar[rd] \ar[ru]\ar@{.}[rr] & & \tau_{_{B_1}}^{-2}X_1 \ar[ru]& & \\
 && X_2 \ar[ru] \ar[r]  \ar[rd] & Z\ar[r] & \tau_{_{B_1}}^{-1}X_2 \ar[ru] && &\\
       & & & M  \ar[ru]& && & }$$

\noindent and finally, in $\Gamma_{A_1} = \Gamma_{({A_1\cap A_2})[X][W]}$ there is a subquiver as follows

$$
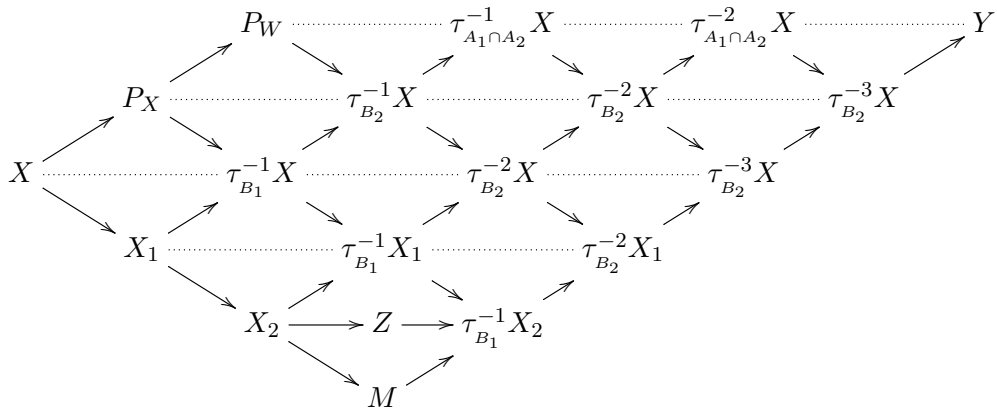
\xymatrix @!0 @R=1.0cm  @C=1.6cm  { & & P_W  \ar[rd]\ar@{.}[rr] & & \tau_{_{A_1\cap A_2}}^{-1}X \ar[rd]\ar@{.}[rr]&   & \tau_{_{A_1\cap A_2}}^{-2}X \ar[rd]\ar@{.}[rr] &    & Y  \\
& P_X  \ar[rd]\ar@{.}[rr] \ar[ru]& & \tau_{_{B_2}}^{-1}X \ar[rd]\ar@{.}[rr]\ar[ru]&   & \tau_{_{B_2}}^{-2}X \ar[rd]\ar@{.}[rr]\ar[ru] &    & \tau_{_{B_2}}^{-3}X \ar[ru] & \\
X \ar[ru] \ar[rd]\ar@{.}[rr]&  & \tau_{_{B_1}}^{-1}X \ar[rd]\ar@{.}[rr] \ar[ru]&   & \tau_{_{B_2}}^{-2}X \ar[rd]\ar@{.}[rr] \ar[ru] &    & \tau_{_{B_2}}^{-3}X \ar[ru] &&\\
   & X_1 \ar[ru]\ar[rd]\ar@{.}[rr]& & \tau_{_{B_1}}^{-1}X_1 \ar[rd] \ar[ru]\ar@{.}[rr] & & \tau_{_{B_2}}^{-2}X_1 \ar[ru]& & & \\
 && X_2 \ar[ru] \ar[r]  \ar[rd] & Z\ar[r] & \tau_{_{B_1}}^{-1}X_2 \ar[ru] &&& &\\
       & & & M  \ar[ru]& && & &}$$
\noindent where there is a sectional path from $X$ to $M$ and from $M$ to $Y$, as we want to prove.

We also can coextend $({A_1\cap A_2})[X]$ in order to get $A_1$. Let $A_1={[0][({A_1\cap A_2})[X]}]$. Then in $\Gamma_{A_1}$ there is a subquiver as follows

$$
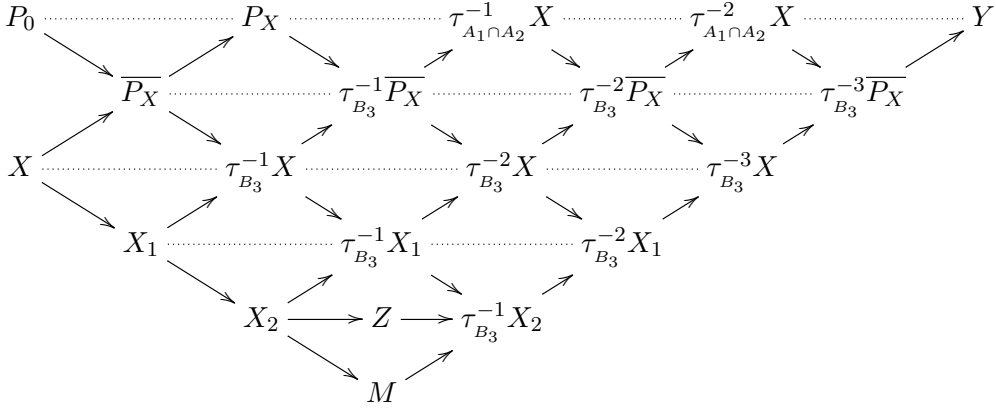
\xymatrix @!0 @R=1.0cm  @C=1.6cm  {  P_0  \ar[rd]\ar@{.}[rr] & & P_X \ar[rd]\ar@{.}[rr]&   & \tau_{_{A_1\cap A_2}}^{-1}X \ar[rd]\ar@{.}[rr] &  & \tau_{_{A_1\cap A_2}}^{-2}X \ar[rd]\ar@{.}[rr]&  & Y  \\
&\overline{ P_X}  \ar[rd]\ar@{.}[rr] \ar[ru]& & \tau_{_{B_3}}^{-1}\overline{ P_X} \ar[rd]\ar@{.}[rr]\ar[ru]&   & \tau_{_{B_3}}^{-2}\overline{ P_X} \ar[rd]\ar@{.}[rr]\ar[ru] &    & \tau_{_{B_3}}^{-3}\overline{ P_X} \ar[ru] & \\
X \ar[ru] \ar[rd]\ar@{.}[rr]&  & \tau_{_{B_3}}^{-1}X \ar[rd]\ar@{.}[rr] \ar[ru]&   & \tau_{_{B_3}}^{-2}X \ar[rd]\ar@{.}[rr] \ar[ru] &    & \tau_{_{B_3}}^{-3}X \ar[ru] &&\\
   & X_1 \ar[ru]\ar[rd]\ar@{.}[rr]& & \tau_{_{B_3}}^{-1}X_1 \ar[rd] \ar[ru]\ar@{.}[rr] & & \tau_{_{B_3}}^{-2}X_1 \ar[ru]& & & \\
 && X_2 \ar[ru] \ar[r]  \ar[rd] & Z\ar[r] & \tau_{_{B_3}}^{-1}X_2 \ar[ru] &&& &\\
       & & & M  \ar[ru]& && & &}$$

\noindent where there is also a sectional path from $X$ to $M$ and from $M$ to $Y$.

With similar arguments as above we can prove that there are sectional paths from $X$ to $M$ and from $M$ to $Y$        in $\Gamma_{A_2}$.
\end{proof}

\begin{teo}\label{E6} Let $A$ be a representation-finite monomial tree algebra  with only one zero-relation $\alpha_1\dots \alpha_m$, with $m \geq 3$, where $s(\alpha_1)=a$ and $e(\alpha_m)=b$. Let $A_1 \cap A_2$ be a hereditary algebra of type $E_6$. Then $r_{A}= r_{A_1} + r_{A_2} - r_{A_1 \cap A_2}.$
\end{teo}
\begin{proof}
Note that if $X\in \mbox{mod}\,(A_1\cap A_2)$ is the first module that we extend $A_1\cap A_2$ to get $ A_1$, then $X$ is a direct summand of $\mbox{rad}\,P_a$ in $\mbox{mod}\,A$. Similarly, if $Y$ is the first module that we extend $A_1\cap A_2$ to get $A_2$, then $Y$ is a direct summand of $I_b/\mbox{soc}\,I_b$.

By Proposition \ref{3} there exists a module $M$ and sectional paths from $X$ to $M$ and from $M$ to $Y$ en $\Gamma_{A_1}$ and in  $\Gamma_{A_2}$. Applying Lemma \ref{1} and Lemma \ref{10}, we have that there are sectional paths from $X$ to $M$ and from $M$ to $Y$ in $\Gamma_A$.

Since $X$ is a direct summand of $\mbox{rad}\,P_a$, by Lemma \ref{9}, $M \notin \mathcal{C}$. On the other hand, since $Y$ is a direct summand of $I_b/\mbox{soc}\,I_b$, again by Lemma \ref{9}, $M \notin \mathcal{A}$. Therefore, $M \in \mathcal{B}$ proving that  $\mathcal{B}\neq \emptyset$.
\end{proof}

We show in an example how to determine the nilpotency index of the radical of a module category.

\begin{ej}\label{ejemplo1}\emph {Consider the algebra $A=kQ_A/I_A$ given by the presentation}

$${\xymatrix   @R=0.4cm  @C=0.6cm {
&&4\ar[d]&&&&& \\
1\ar@{--}@/_{5mm}/[rrrrr]\ar[r]^{\alpha}&2\ar[r]^{\beta}&3\ar[r]^{\gamma}&5\ar[r]^{\delta}&6\ar[r]^{\lambda}&7&&
}}$$
\vspace{0.1in}

\noindent \emph{where $I_A=<\lambda\delta\gamma\beta\alpha>.$ }

\emph {If we proceed to delete the arrow $\lambda$, we get the hereditary algebra $A_1$ whose quiver is}

$$\begin{array}{cc}
  {\xymatrix   @R=0.3cm  @C=0.6cm {
&&4\ar[d]&&\\
1\ar[r]&2\ar[r]&3\ar[r]&5\ar[r]&6
}} \end{array}$$
\vspace{0.1in}

\noindent \emph{and by Theorem  \ref{ppalher}, $r_{A_1}= 11$. The algebra $A_2$ is obtained by deleting the arrow $\alpha$}

$$\begin{array}{cc} {\xymatrix   @R=0.3cm  @C=0.6cm {
&4\ar[d]&&&\\
2\ar[r]&3\ar[r]&5\ar[r]&6\ar[r]&7
}}\end{array}$$

\noindent \emph{where by Theorem  \ref{ppalher}, $r_{A_2}=2.6 -3 = 9$.  Finally, the intersection algebra $A_1 \cap A_2$ is}

$${\xymatrix   @R=0.3cm  @C=0.6cm {
&4\ar[d]&&\\
2\ar[r]&3\ar[r]&5\ar[r]&6
}}$$
\noindent \emph{where $r_{A_1 \cap A_2}=7$. By Proposition \ref{interDn}, we get that  $r_A= 11+9-7=13$.}
\end{ej}

Finally,  we show how to read the nilpotency index of some tree algebras with $r$ zero-relations not overlapped.

Let $A \simeq kQ_A/I_A$ be a representation-finite tree algebra with $r$ zero-relations $\alpha_{m_i}^{i} \dots \alpha_{1}^{i}$ for $i= 1, \dots, r$, not overlapped.
Assume that $Q_A$ is a gluing of quivers as the ones described in Proposition \ref{ejemplos} (1), (2), (3) and in Proposition \ref{interDn}. More precisely, the quiver of $Q_A$ is as follows:

$$
\xymatrix @!0 @R=1.0cm  @C=1.0cm  {
Q_{A_1}\ar@^{-}[r] & \bullet \ar@^{-}[r]^{\alpha_{1}^{1}} & \ar@{.}[r]  &  \ar@^{-}[r]^{\alpha_{m_1}^{1}}& \bullet \ar@^{-}[r] &\ar@^{-}[r] Q_{A_2}& \bullet \ar@^{-}[r]^{\alpha_{1}^{2}} & \ar@{.}[r]  &  \ar@^{-}[r]^{\alpha_{m_2}^{2}}& \bullet \ar@^{-}[r] &\ar@^{-}[r] Q_{A_3}& \ar@{.}[r]  &\\
&\ar@^{-}[r] & Q_{A_r} \ar@^{-}[r] &\bullet \ar@^{-}[r]^{\alpha_{1}^{r}} & \ar@{.}[r]  &  \ar@^{-}[r]^{\alpha_{m_r}^{r}}& \bullet \ar@^{-}[r] & Q_{A_{r+1}}  \\}$$

\noindent where for each $i=1, \dots, r$

$$
\xymatrix @!0 @R=1.0cm  @C=1.6cm  {
Q_{A_i}\ar@^{-}[r] & \bullet \ar@^{-}[r]^{\alpha_{1}^{i}} & \ar@{.}[r]  &  \ar@^{-}[r]^{\alpha_{m_i}^{i}}& \bullet \ar@^{-}[r] & Q_{A_{i+1}} &}$$

\noindent is a  quiver as the ones described in Proposition \ref{ejemplos}  (1), (2), (3) and in Proposition \ref{interDn} and where the zero-relations are  precisely the ones described in such propositions.

Let $B_1$ to be the algebra where $Q_{B_1}$ is obtained from $Q_A$ by deleting the arrow $\alpha_{m_{2}^{2}}$,   and containing the  relation  $\alpha_{m_1}^{1}\dots \alpha_{1}^{1}$. Let  $B_i$ for $i=2, \dots, r-1$ to be the algebra where $Q_{B_i}$ is obtained from $Q_A$ by deleting the arrow $\alpha_{1}^{i-1}$ and  containing  the relation  $\alpha_{m_i}^{i}\dots \alpha_{1}^{i}$ up to the edge ${\alpha_{{m_{i+1}-1}}^{i+1}}$ and finally, let the algebra $B_r$ to be the one where $Q_{B_r}$ is obtained from $Q_A$ by deleting the arrow $\alpha_{m_{r-1}^{r-1}}$ and containing  the relation  $\alpha_{m_r}^{r} \dots \alpha_{1}^{r}$.
\vspace{.05in}

With this notations we are in position to prove the last theorem.

\begin{teo} \label{m-relations} Let $A \simeq kQ_A/I_A$ be a representation-finite tree algebra with $r$ zero-relations not overlapped, and  $r\geq 2$. Consider the algebras $B_i$ described as above, for $i=1, \dots,  r$. Then $r_A= \emph{max} \{ r_{_{B_i}}\}_{i=1}^{r}$.
\end{teo}

\begin{proof} Let $r=2$.  By Corollary \ref{coro-vertex}, if $a$ and $b$ are vertices involved in different
zero-relations then $r_A=\mbox{max}\,\{\ell_A(P_{a} \rightsquigarrow I_{a}),\ell_A(P_{b} \rightsquigarrow I_{b})\}+1$. Consider $\rho_1=\alpha_{m_1}^{1} \dots \alpha_{1}^{1}$  and $\rho_2=\alpha_{m_2}^{2} \dots \alpha_{1}^{2}$ the given zero-relations, not overlapped.

We consider $A$ to be the pullback algebra of $B_1 \twoheadrightarrow B_1 \cap C_1$ and  $C_1 \twoheadrightarrow B_1 \cap C_1$  where $B_1$ is the algebra whose bound quiver is  obtained from $Q_A$ by deleting the arrow $\alpha_{1}^{1}$ and containing  the relation  $\rho_2$
and where  $C_1$ is the algebra whose quiver is obtained from $Q_A$ by deleting the arrow $\alpha_{m_1}^{1}$ and containing the arrow $\alpha_1^1$.

Similarly, we consider that $A$ is the pullback algebra of $B_2 \twoheadrightarrow B_2 \cap C_2$ and  $C_2 \twoheadrightarrow B_2 \cap C_2$  where $B_2$ is the algebra whose bound quiver is obtained from $Q_A$ by deleting the arrow $\alpha_{1}^{2}$ and containing  the relation  $\rho_1$  and  where $C_2$ is the algebra whose quiver is obtained from $Q_A$ by deleting the arrow $\alpha_{m_2}^{2}$ and containing the arrow $\alpha_1^2$.

By Lemma \ref{proyiny}, $P_{a}$ and $I_{a}$ in $\mbox{mod}\,B_2$ are also the projective and the injective in the vertex $a$ in $\mbox{mod}\, A$.

We claim that  $\ell(P_{a} \rightsquigarrow I_{a})$ in $\mbox{mod}\, A$ coincides with  $\ell(P_{a} \rightsquigarrow I_{a})$ in $\mbox{mod}\, B_2$. Indeed,  since $I_{a}$ is a direct summand of $D(B_2)'$, then $I_{a}\in  \mathcal{A}_2$, where $\mathcal{A}_2 = \mbox{Pred}\, (D(B_2))'$. Moreover, $P_{a}$ is also in $\mathcal{A}_2$ because $\mbox{Hom}_A(P_{a}, I_{a})\neq 0.$
Because of  the zero-relations  that we are considering then by Proposition \ref{4} we have that $\mathcal{A}_2\cap \mathcal{C}_2=\emptyset$. Thus
all the  paths from $P_{a}$ to $I_{a}$ are in $\Sigma'_{B_2}$. Hence, our claim
follows from Lemma \ref{sigmas}. Moreover, $B_2$ is an algebra with only one zero-relation and therefore $r_{B_2}=\ell(P_{a} \rightsquigarrow I_{a})+1$.

Similarly, we get that $r_{{B_1}}=\ell(P_{b} \rightsquigarrow I_{b})+1.$ In consequence, $r_A= \mbox{max}\, \{ r_{{B_1}},  r_{{B_2} }\}$.

Iterating the above process  to the algebras $B_i$ and $C_i$ as we defined above, we get the result for any $r\geq 2$.\end{proof}

We end up this paper computing $r_A$ for two different algebras.

\begin{ej}
\emph{Consider the algebra $A=Q_A/I_A $ given by the presentation:}

$$
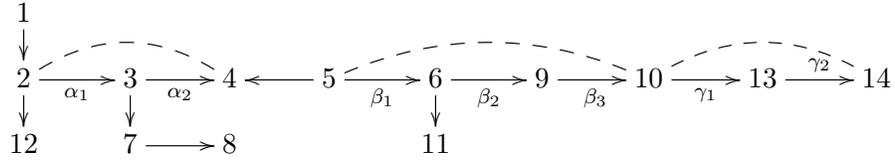
{\xymatrix   @R=0.4cm  @C=0.9cm {
1\ar[d]&&&&&&&&\\
2\ar[d]\ar[r]_{\alpha_1}\ar@{--}@/^{5mm}/[rr]&3\ar[d]\ar[r]_{\alpha_2}&4&5\ar[r]_{\beta_1}\ar[l]\ar@{--}@/^{5mm}/[rrr]&6\ar[d]\ar[r]_{\beta_2}
& 9\ar[r]_{\beta_3}&10\ar@{--}@/^{5mm}/[rr]\ar[r]_{\gamma_1}&13\ar[r]^{\gamma_2}&14\\
12&7\ar[r]&8&&11&&&&
}}$$
\vspace{0.1in}

\noindent {\emph{with $I_{A}=<\alpha_2\alpha_1,\beta_3\beta_2\beta_1, \gamma_2\gamma_1>.$}}
\vspace{0.1in}

\emph{Let $B_1$ be the algebra given by the quiver}

$$
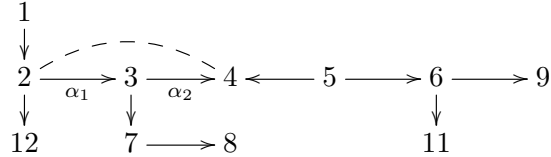
{\xymatrix   @R=0.4cm  @C=0.9cm {
1\ar[d]&&&&&\\
2\ar[d]\ar[r]_{\alpha_1}\ar@{--}@/^{5mm}/[rr]&3\ar[d]\ar[r]_{\alpha_2}&4&5\ar[r]\ar[l]&6\ar[d]\ar[r]
& 9\\
12&7\ar[r]&8&&11&
}}$$

\noindent {\emph {with $I_{B_1}=<\alpha_2\alpha_1>$, $B_2$ given by the quiver}}
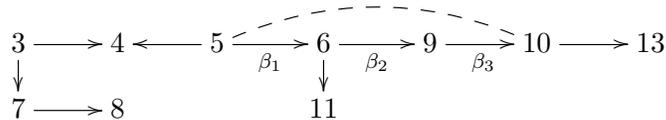
\vspace{0.1in}

$${\xymatrix   @R=0.4cm  @C=0.9cm {
&&&&&\\
3\ar[d]\ar[r]&4&5\ar[r]_{\beta_1}\ar[l]\ar@{--}@/^{5mm}/[rrr]&6\ar[d]\ar[r]_{\beta_2}
& 9\ar[r]_{\beta_3}&10\ar[r]&13\\
7\ar[r]&8&&11&&&
}}$$
\vspace{0.1in}

\noindent \emph {with $I_{B_2}=<\beta_3\beta_2\beta_1>$  and $B_3$ given by the quiver}

$$
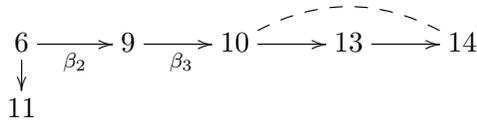
{\xymatrix   @R=0.4cm  @C=0.9cm {
&&&\\
6\ar[d]\ar[r]_{\beta_2}
& 9\ar[r]_{\beta_3}&10\ar@{--}@/^{5mm}/[rr]\ar[r]&13\ar[r]&14\\
11&&&&
}}$$
\emph{with $I_{B_3}=<\gamma_2\gamma_1>.$ Again, by Theorem \ref{m-relations}, $r_A = \mbox{max}\, \{r_{B_1},r_{B_2},r_{B_3}\}$. Moreover, by Proposition \ref{ejemplos}, we have  that $r_{B_1}=19$, $r_{B_2}=14$ and $r_{B_3}=6$,
and therefore $r_A =19$.}
\end{ej}

\end{document}